\documentclass[aos]{imsart}

\RequirePackage{amsthm,amsmath,amsfonts,amssymb}
\RequirePackage[numbers]{natbib}
\RequirePackage[colorlinks,citecolor=blue,urlcolor=blue]{hyperref}
\RequirePackage{graphicx}

\startlocaldefs
\numberwithin{equation}{section}
\theoremstyle{plain}
\newtheorem{lemma}{Lemma}[section]
\newtheorem{theorem}[lemma]{Theorem}
\newtheorem{proposition}[lemma]{Proposition}

\theoremstyle{remark}
\newtheorem{example}[lemma]{Example}
\newtheorem{remark}[lemma]{Remark}
\newtheorem{definition}[lemma]{Definition}
\usepackage{ifthen}
\usepackage{dsfont}
\usepackage{amssymb,latexsym}
\usepackage{subfigure}

\newcommand{\wt }{\widetilde}

\newcommand{\bfR}{{\mathbf R}}

\newcommand{\bfX}{{\mathbf X}}

\newcommand{\bfx}{{\mathbf x}}

\newcommand{\la}{\lambda}

\newcommand{\beao}{\begin{eqnarray*}}
\newcommand{\eeao}{\end{eqnarray*}}
\newcommand{\beam}{\begin{eqnarray}}
\newcommand{\eeam}{\end{eqnarray}}
\newcommand{\barr}{\begin{array}}
\newcommand{\earr}{\end{array}}

\newcommand{\tI}{\widetilde I}
\newcommand{\bco}{\begin{corrolary}}
\newcommand{\eco}{\end{corrolary}}

\newcommand{\E}{\mathbb{E}}
\renewcommand{\P}{\mathbb{P}}
\newcommand{\1}{\mathds{1}}
\newcommand{\R}{\mathbb{R}}
\newcommand{\N}{\mathbb{N}}

\newcommand{\bfS}{{\mathbf S}}

\newcommand{\Var}{\operatorname{Var}}

\newcommand{\Frechet}{Fr\'{e}chet }

\DeclareMathOperator{\e}{e}

\newcommand{\X}{{\mathbf X}}
\newcommand{\Y}{{\mathbf Y}}

\newcommand{\dint}{\,\mathrm{d}}

\newcommand{\vep}{\varepsilon}
\newcommand{\nto}{n \to \infty}

\newcommand{\rhs}{right-hand side}
\newcommand{\ts}{time series}

\newcommand{\tr}{\operatorname{tr}}

\renewcommand{\path}{\operatorname{path}}
\newcommand{\partition}{\operatorname{partition}}

\newcommand{\diag}{\operatorname{diag}}

\newcommand{\MP}{Mar\v cenko--Pastur }
\newcommand{\runs}{\operatorname{runs}}
\newcommand{\simples}{\operatorname{simples}}

\newcommand{\cas}{\stackrel{\rm a.s.}{\rightarrow}}

\newcommand{\eid}{\stackrel{d}{=}}

\newcommand{\ahmplaw}{$\alpha$-heavy MP law}
\newcommand\iv[1]{{[\![ 1,#1]\!]}}  

\endlocaldefs

\begin{document}

\begin{frontmatter}
\title{Limiting distributions for eigenvalues of sample
correlation matrices from  heavy-tailed populations}
\runtitle{Eigenvalue distributions of heavy-tailed correlation matrices}

\begin{aug}
\author[A]{\fnms{Johannes} \snm{Heiny}\ead[label=e1]{johannes.heiny@rub.de}} 
\and
\author[B]{\fnms{Jianfeng} \snm{Yao}\ead[label=e2]{jeffyao@hku.hk}}
\address[A]{Department of Mathematics, Ruhr University Bochum, \printead{e1}}

\address[B]{Department of Statistics and Actuarial Science, The University of Hong Kong. \printead{e2}}
\end{aug}

\begin{abstract}
Consider a $p$-dimensional population $\bfx\in\R^p$
    with iid coordinates that are regularly varying with index $\alpha\in (0,2)$. 
    Since the variance of $\bfx$ is infinite, the diagonal elements of the sample covariance matrix
    $\bfS_n=n^{-1}\sum_{i=1}^n {\bfx_i}\bfx'_i$ based on a sample $\bfx_1,\ldots,\bfx_n$ from the population tend to infinity as $n$ increases and it
    is of interest to use instead the sample correlation matrix
    $\bfR_n= \{\diag(\bfS_n)\}^{-1/2}\, \bfS_n\{\diag(\bfS_n)\}^{-1/2}$.
    This paper finds the limiting distributions of the eigenvalues of
    $\bfR_n$ when both the dimension $p$ and the sample size $n$ grow to
    infinity such that  $p/n\to \gamma \in (0,\infty)$.
    The family of limiting distributions $\{H_{\alpha,\gamma}\}$ is
    new and depends on the two parameters $\alpha$ and
    $\gamma$.
  The moments of $H_{\alpha,\gamma}$ are fully identified as sum of two
  contributions: the first from  the classical \MP law 
  and a second due to heavy tails.
  Moreover, the family $\{H_{\alpha,\gamma}\}$ has  continuous
  extensions  at the  boundaries  $\alpha=2$ and $\alpha=0$ leading to
   the \MP law and a modified Poisson distribution, respectively.

  Our proofs use  the method of moments, the path-shortening algorithm
  developed in \cite{heiny:mikosch:2017:corr} and some  novel graph
  counting combinatorics. As a consequence, the moments of
  $H_{\alpha,\gamma}$ are expressed in terms of combinatorial objects
  such as Stirling numbers of the second kind. A simulation study on
  these limiting distributions $H_{\alpha,\gamma}$ is also provided for
  comparison with the \MP law.
\end{abstract}

\begin{keyword}[class=MSC]
\kwd[Primary ]{60B20}
\kwd[; secondary ]{60F05 60G10 60G57 60G70}
\end{keyword}

\begin{keyword}
\kwd{Sample correlation matrix}
\kwd{limiting spectral distribution}
\kwd{method of moments}
\kwd{infinite variance}
\kwd{\MP~law}
\kwd{stable distribution}
\end{keyword}

\end{frontmatter}

\section{Introduction}\label{sec:1}\setcounter{equation}{0}

  Consider a $p$-dimensional population $\bfx=(X_1,\ldots,X_p)\in\R^p$
  where the coordinates  $X_i$  are independent non-degenerated
  random variables and  identically distributed as a centered random variable $\xi$. 
  For a sample $\bfx_1,\ldots,\bfx_n$ from the population we construct the 
  data matrix $\X=\X_n=(\bfx_1,\ldots,\bfx_n)=(X_{ij})_{1\le    i\le p; 1\le j  \le n}$, the sample covariance matrix $\bfS$
  and the sample correlation matrix $\bfR$ as follows:
  \begin{equation*}
	\begin{split}
  \bfS &= \bfS_n =\frac1n \sum_{i=1}^n {\bfx_i}\bfx'_i =\frac1n \X\X'\,,
  \\
  \bfR &=\bfR_n =\{\diag(\bfS_n)\}^{-1/2}\, \bfS_n\{\diag(\bfS_n)\}^{-1/2}= \Y \Y'\,.
	\end{split}
  \end{equation*}
  Here the standardized  matrix $\Y=\Y_n=(Y_{ij})_{1\le    i\le p; 1\le j  \le n}$ for the correlation matrix has
  entries 
  \beam\label{def:R}
  Y_{ij}=Y_{ij}^{(n)}=\frac{X_{ij}}{\sqrt{X_{i1}^2+\cdots+X_{in}^2}}\,,
  \eeam
  which depend on $n$. Throughout the paper, we often suppress the dependence on $n$ in our notation.

  Both the sample covariance matrix $\bfS$ and
  the sample correlation matrix $\bfR$ are
  fundamental tools in
  multivariate statistical analysis such as 
  PCA, canonical correlation analysis, classification  or hypothesis testing on
  population covariance matrix \cite{And03}.
  A large amount of recent literature is devoted to their study in a
  high-dimensional scenario where
  $p$ and $n$ are of comparable magnitude. We  consider the
  asymptotic regime
  \begin{equation}\label{Cgamma}
    p=p_n \to \infty \quad \text{ and } \quad \frac{p}{n}\to \gamma\in (0,\infty)\,,\quad \text{ as } \nto\,. \tag{$C_\gamma$}
  \end{equation}
Random  matrix theory (RMT) has
  provided relevant tools in this perspective, see \cite{yao:zheng:bai:2015} for a
  recent synthesis.
	
  Recall that if ${\bf A}$ is a matrix with $p$ real eigenvalues
  $\la_{1}( {\bf A} ) \ge \cdots \ge\la_{p}(  {\bf A})$, 
  its  {\em empirical spectral distribution}
  (ESD)
  is the normalized counting measure of the  eigenvalues, that is
  $  F_{\bf A}= p^{-1} \sum_{i=1}^p  \delta_{ \lambda_i(\bf A)}$.
  In the finite variance case with  $\E\xi^2<\infty$,
  the spectral properties of the sample covariance matrix $\bfS$
  have been well studied in RMT since the pioneering work
  \cite{marchenko:pastur:1967} where it is shown that $F_{\bfS}$ converges weakly to the
  celebrated \MP (MP) law.
  Subsequent developments include several  ground-breaking results such as  the
  convergence of the largest eigenvalue $\lambda_1(\bf S)$ and the smallest eigenvalue $\lambda_p(\bfS)$ to the edges of  the MP law
  \citep{BaiYin88a,tikhomirov:2015},
  asymptotic  normality of linear spectral statistics of $\bfS$ \citep{BS04},  or its edge
  universality towards the Tracy-Widom law \citep{johnstone:2001,Peche2012,PillaiYin2014}. Apart from the convergence of $\lambda_p(\bfS)$ all those results require a finite fourth moment $\E \xi^4$.
	
	If $\E \xi^4=\infty$, the theory for the eigenvalues and eigenvectors of $\bfS$ is quite different from the classical \MP~theory which applies in the light-tailed case. For example, if $\xi$ is regularly varying with index $\alpha\in (0,4)$, that is 
\beam\label{eq:regvar}
\P(|\xi|>x)=  L(x)\, x^{-\alpha}\,,\qquad x>0\,,
\eeam
for a function $L$ that is slowly varying at infinity, then the properly normalized largest eigenvalue of $\bfS$ converges to a \Frechet distribution with parameter $\alpha/2$. A detailed account on the developments in the heavy-tailed case can be found in \cite{auffinger:arous:peche:2009,basrak:heiny:jung:2020,davis:heiny:mikosch:xie:2016,heiny:mikosch:2017:iid,heiny:mikosch:2019:bernoulli,heiny2020large,soshnikov:2004,soshnikov:2006}.
The limiting spectral distribution in the infinite variance case $\E\xi^2 =\infty$ was found in \cite[Theorem~1.10]{belinschi:dembo:guionnet:2009} 
and \cite[Theorem~1.6]{arous:guionnet:2008}.  Under \eqref{Cgamma} and assuming $\xi$ is regularly varying with index $\alpha\in (0,2)$, they proved that the empirical spectral distribution of the suitably normalized $\bfS$ converges weakly to a probability measure with infinite support that depends on the parameters $\alpha$ and $\gamma$.

  
  In contrast,
  the study of the high-dimensional sample correlation matrix $\bfR$
  is more recent and more limited. A fundamental reason is that compared
  to the original data matrix $\X$, the
  entries ${Y_{ij}}$ of the standardized matrix $\Y$ are no longer independent within
  the same row (the different rows remain independent identically distributed (iid)). This makes the
  correlation matrix more challenging to study. 
  Jiang \cite{Jiang2004b} first established that if $\E[\xi^2]<\infty$, 
  the ESD $F_{\bfR}$ also converges weakly to the MP law. Jiang \cite{jiang:2004b} also analyzed the asymptotic distribution of the largest off-diagonal entry of $\bfR$ and proved that (suitably standardized) $\max_{i<j} |R_{ij}|$ tends to a Gumbel distributed random variable. Later \cite{zhou:2007} found the necessary assumption $\E|\xi|^{6-\vep}<\infty$ for the Gumbel limit and \cite{heiny2020point} studied the point process of all off-diagonal entries.	
  When  $\xi$ has a  subexponential tail (which implies the existence of
  moments of all orders),  edge universality towards the Tracy-Widom law
  was established
  for  the sample correlation matrix 
  $\bfR$ in  \cite{Bao2012,Pillai2012}.
  Among recent 
  developments,  a  central limit theorem
  for linear spectral statistics of  $\bfR$ was
  established in \cite{Gao2017}
  under the finite fourth moment condition $\E[\xi^4]<\infty$, and
  \cite{Zheng2019} proved asymptotic normality of 
  a series of test statistics for 
  one-, two-  or multiple sample hypotheses on
  population correlation matrices.
  
  One common  feature shared by these recent developments
  on sample correlation matrix $\bfR$ is that under the finite second
  moment condition $\theta=\E[\xi^2]<\infty$, the normalizing  denominator 
  $S_{ii}=\{X_{i1}^2+\cdots+X_{in}^2\}/n$ in the definition~\eqref{def:R}  of
  $Y_{ij}$ almost surely converges to $\theta$ by the law of large
  numbers.  
	By Lemma 2 in \cite{bai:yin:1993}, the ``uniform  approximation'' $\max_{i} |S_{ii} -\theta | \cas 0$ is equivalent to $\E\xi^4<\infty$. Then Weyl's eigenvalue perturbation inequality yields that $\max_i|\lambda_i(\bfR)-\theta^{-1} \lambda_i(\bfS)|\cas 0$. As a consequence, the spectral properties of $\bfR$ and $\bfS$ are asymptotically equivalent.
This has been generalized to population correlation matrices with uniformly bounded spectrum in \cite[Theorem~1]{elkaroui:2009}.
	Therefore,  a main step in the above references on 
  the correlation
  matrix $\bfR$ relies on a 
  precise estimate of the error in the above approximation.
  For example, in
  \cite{Jiang2004b,Bao2012,Pillai2012}, this approximation
  error is  shown to
  be negligible and the results obtained for $\bfR$ are the same as
  those known for $\bfS$.
  In this paper, we study the infinite variance case with
  $\theta=\E[\xi^2]=\infty$.  This approximation argument breaks down.
  Indeed, it will be shown that the various limits of $\bfR$ are not
  anymore related to their counterparts for $\bfS$.

Particularly, a refinement of the result of \cite{Jiang2004b}  
  is proposed in \cite{heiny:mikosch:2017:corr}.
Assume for a moment that $\xi$ is symmetrically distributed,  that is,
 $\xi \eid -\xi$.
Theorem 3.1 in \cite{heiny:mikosch:2017:corr} shows that if
\begin{equation}\label{eq:ny4}
\lim_{\nto} n\,\E [ Y_{11}^4 ] =0\,,
\end{equation}
then $F_{\bfR}$ converges weakly almost surely to $\sigma_{MP,\gamma}$,
the Mar\v{c}enko-Pastur law with index $\gamma>0$  (see   \eqref{MPlaw}).
Conversely, if condition \eqref{eq:ny4} does not hold, i.e.,
$\liminf_{\nto} n\,\E [ Y_{11}^4 ] >0$,
then 
$\liminf_{\nto} \E\big[\int x^k dF_{\bfR_n}(x)\big]> \int x^k d\sigma_{MP,\gamma}(x)$ for 
$k\ge 4$.
Therefore \eqref{eq:ny4} is a necessary and sufficient condition for the convergence of $F_{\bfR_n}$ to the MP law.

In \cite{gine:goetze:mason:1997} it was proved that condition \eqref{eq:ny4} holds if the distribution of 
$\xi$ is in the domain of attraction of the normal law, which is equivalent to the function $f(x)=\E[\xi^2 \1_{\{|\xi|\le x\}}]$ being slowly varying. They also derive the formula
\begin{equation*}
\E[ Y_{11}^4 ] =  \int_0^{\infty} 
 t (\E[\e^{-t \xi^2}])^{n-1}   \E[\xi^{4}\e^{-t \xi^2} ] \, \dint t\,.
\end{equation*}

Our focus is on the case where condition \eqref{eq:ny4} is violated
(which in particular implies $E\xi^2=\infty$).
Proposition~1 in \cite{mason:zinn:2005} asserts that $\xi$ is regularly varying with index $\alpha\in (0,2)$ if and only if
\begin{equation}\label{eq:limita<2}
\lim_{\nto} n\,\E[Y_{11}^4] =1-\frac{\alpha}{2}\,.
\end{equation}
Hence, \eqref{eq:ny4} does not hold if $\xi$ is regularly varying with index $\alpha\in (0,2)$.

\subsection*{About this paper} 
As seen in the above discussion, the sample correlation matrix $\bfR$ has mainly been studied under the finite fourth moment assumption. In the intermediate regime, where $\E\xi^4=\infty$ and $\E\xi^2<\infty$, the limiting spectral distribution is known to be the MP law and \cite{heiny:mikosch:2017:corr} studied the extreme eigenvalues. Under infinite variance, the limiting spectral distribution of $\bfS$ has been characterized, whereas no results on the sample correlation matrix $\bfR$ seem to be available in the literature.

  By assuming that the distribution of $\xi$ is symmetric and 
  regularly varying with index
  $\alpha\in(0,2)$, we establish in this paper that the sequence of
  ESDs $F_{\bfR}$ converges weakly to a new distribution $H_{\alpha,\gamma}$
  termed as \ahmplaw\ with parameter $\gamma$. This result is
  introduced in Section~\ref{ssec:hMP} (Theorem~\ref{thm:mainsimplified})
  where  comparison with  the MP law $\sigma_{MP,\gamma}$ is also
  proposed. Theorem~\ref{thm:limit} shows that the class of distributions $H_{\alpha,\gamma}$ can be
  extended 
  continuously  at the 
  boundaries $\alpha=2$ and $\alpha=0$,  
  yielding the MP law $\sigma_{MP,\gamma}$ and a modified Poisson
  distribution, respectively.  
 Subsequently  in Section~\ref{ssec:estimate}, we propose a
    consistent estimator for the tail index $\alpha$.

 The remaining sections of the paper are devoted to the proofs of
  Theorems \ref{thm:mainsimplified} and \ref{thm:limit}. Our main tool is a moment method that required a
  specific and careful counting of relevant graphs  to cope with exploding second moments of the
  matrix entries $\{X_{ij}\}$. Section~\ref{sec:moments} presents the
  main steps of this
   moment method based on a {\em path-shortening algorithm} that was developed in \cite{heiny:mikosch:2017:corr}.
  Section~\ref{sec:F(I)} establishes the combinatorics on associated
  graph counting for the moment method by using set partitions.
  The proof of our main result Theorem~\ref{thm:mainsimplified} is then  completed
  in Section~\ref{sec:mainresult}, which also contains the formula for the moments of the \ahmplaw s $H_{\alpha,\gamma}$.
  Finally, Section~\ref{sec:boundary} proves Theorem~\ref{thm:limit}.

\section{Main results}\label{sec:main}

\subsection{The family of $\alpha$-heavy MP laws}\label{ssec:hMP}

Recall that for $\gamma >0$ the \MP law $\sigma_{MP,\gamma}$ is 
\begin{equation}
  \label{MPlaw}
  \sigma_{MP,\gamma}(\dint x) = f_{\gamma}(x)\dint x + (1-\gamma^{-1}) \1_{(1,\infty)}(\gamma) \delta_0(\dint x),
\end{equation}
where the density of the absolutely continuous part is 
\begin{equation*}
  f_\gamma(x) =
  \frac{\sqrt{( b_\gamma-x)(x-a_\gamma )}}{2\pi \gamma x} \1_{
      [a_\gamma, b_{\gamma}]}(x)\,,\quad  x\in\R, 
\end{equation*}
with $a_\gamma=(1-\sqrt{\gamma})^2 $  and $b_\gamma=(1+\sqrt{\gamma})^2$. 
Its moments are
\begin{equation}\label{eq:momentsmp}
  \beta_k(\gamma)=\int_{a_{\gamma}}^{b_{\gamma}} x^k d\sigma_{MP,\gamma}(x)=\sum_{r=1}^{k} \frac{1}{r} \binom{k}{r-1}\binom{k-1}{r-1}\gamma^{r-1}\,,\quad k\ge 1\,.
\end{equation}.  

In this paper we find a family of new distributions
$\{H_{\alpha,\gamma}\}$ for parameters $\alpha\in(0,2)$ and
$\gamma>0$. We call $H_{\alpha,\gamma}$ the {\em \ahmplaw\ with
  parameter $\gamma$}. Each  $H_{\alpha,\gamma}$ is entirely
determined by its moment sequence
$\mu_{k}(\alpha,\gamma)=\int x^k\dint H_{\alpha,\gamma}(x)$, $k\ge 1$.
The exact expression for   $\mu_k(\alpha,\gamma)$
requires a considerable amount of additional notation: it is given in
\eqref{mainresult}. 
Roughly speaking, $\mu_k(\alpha,\gamma)$ can be decomposed into a \MP part and a heavy tail part as follows
\begin{eqnarray}\label{eq:dec}
\mu_k(\alpha,\gamma) =
\left\{\begin{array}{ll}
\beta_k(\gamma) \,, & \mbox{if } k=1,2,3, \\
\beta_k(\gamma) + d_k(\alpha,\gamma) \,, & \mbox{if } k\ge 4,
\end{array}\right. 
\end{eqnarray}
where $d_k(\alpha,\gamma)>0$ is given in \eqref{mainresult}.
Formula \eqref{mainresult} is explicit and requires some counting that
can be implemented  using computing  software.
For small values of $k$, $d_k(\alpha,\gamma)$ can be evaluated directly.  In Section \ref{sec:k45}, we derive that
\begin{equation}\label{eq:d45}
d_4(\alpha,\gamma)=(1-\alpha/2)^2 \gamma \quad \text{and} \quad d_5(\alpha,\gamma)=(1-\alpha/2)^2 (5 \gamma + 5 \gamma^2)\,.
\end{equation}

The following theorem is the main result of the paper.

\begin{theorem}\label{thm:mainsimplified}
  Assume \eqref{Cgamma} and that $\xi$ is regularly varying with
  index $\alpha \in (0,2)$ and $\xi\eid -\xi$. Then, as $\nto$, the ESDs $F_{\bfR_n}$
  converge weakly in probability to  $H_{\alpha,\gamma}$, the $\alpha$-heavy MP law with
  parameter $\gamma$.
\end{theorem}


The symmetry requirement on the distribution of $\xi$ is
  technical. It allows to neglect all expectations of odd powers of
  matrix entries in our moment method. Since the moment formula \eqref{moment} which is a key ingredient of the proof only depends on the distribution of $\xi^2$ (and not $\xi$), the symmetry restriction can likely be removed and Theorem \ref{thm:mainsimplified}
 also holds for non-symmetrically distributed $\xi$; see Remark~\ref{rem111} for details. \smallskip

We now give some illustrations of the theorem
  and compare the limiting  $\alpha$-heavy MP
  laws $H_{\alpha,\gamma}$ with the classical MP laws. 
\begin{figure}[htb!]
  \centering
    \includegraphics[trim = 0.55in 2.5in 0.45in 2.3in, clip, scale=0.75]{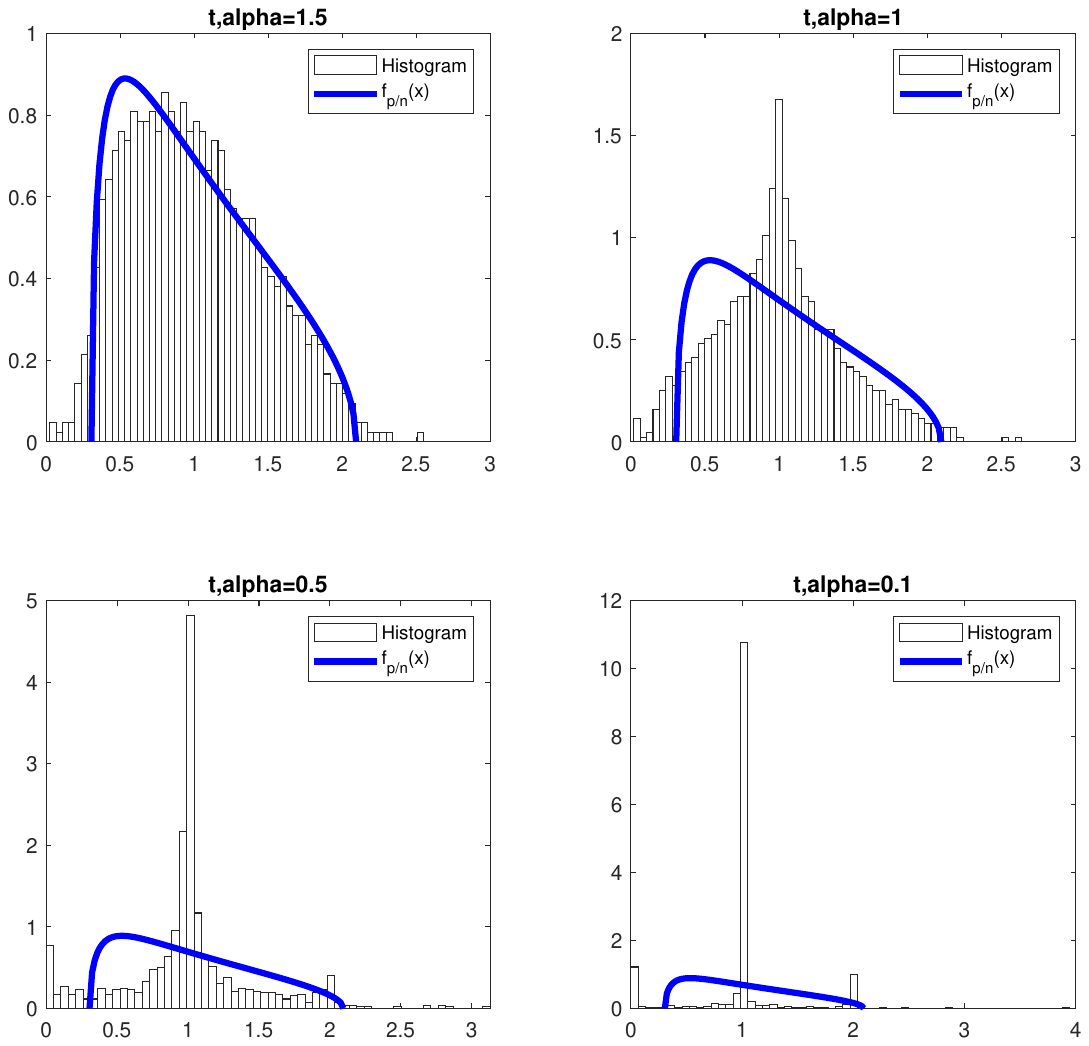}
\caption{Histograms of the \ahmplaw~with parameter $p/n$. The data are
  simulated from a $t$-distribution with different degrees of
    freedom $\alpha\in\{1.5,1,0.5,0.1\}$. The dimension is $p=1000$ and the sample size  $n=5000$.}
\label{fig:fouralphas}
\end{figure}
Figure \ref{fig:fouralphas} shows the shape of $H_{\alpha,\gamma}$ for
different values of $\alpha$, $p=1000, n=5000$ and $\gamma=p/n$. 
The entries $X_{ij}$ were drawn from a $t$-distribution with
$\alpha\in\{1.5,1,0.5,0.1\}$ degrees of freedom. We compare the
(normalized) histogram of the eigenvalues $(\lambda_i(\bfR))$ with the
\MP density $f_{p/n}(x)$. 
The parameter $\gamma=p/n=0.2$ is the same in the four plots in Figure~\ref{fig:fouralphas}
as well as the four MP densities despite their visual difference due
to different scales used in the plots.
Observe that for $\alpha=1.5$ (top left panel) the histogram resembles
$f_{p/n}$ at first sight. At closer inspection one notices that
the \ahmplaw~has a  larger support than the MP law. Moreover, more
mass is concentrated around the mean 1. These two effects become more
pronounced if the tail heaviness of $\xi$ increases, i.e. $\alpha$
decreases. The plots show that most mass is concentrated  around  1 if $\alpha$ is small. 
\begin{figure}[htb!]
  \centering
    \includegraphics[trim = 3.0in 0.5in 0.45in 0.3in, clip, scale=0.23]{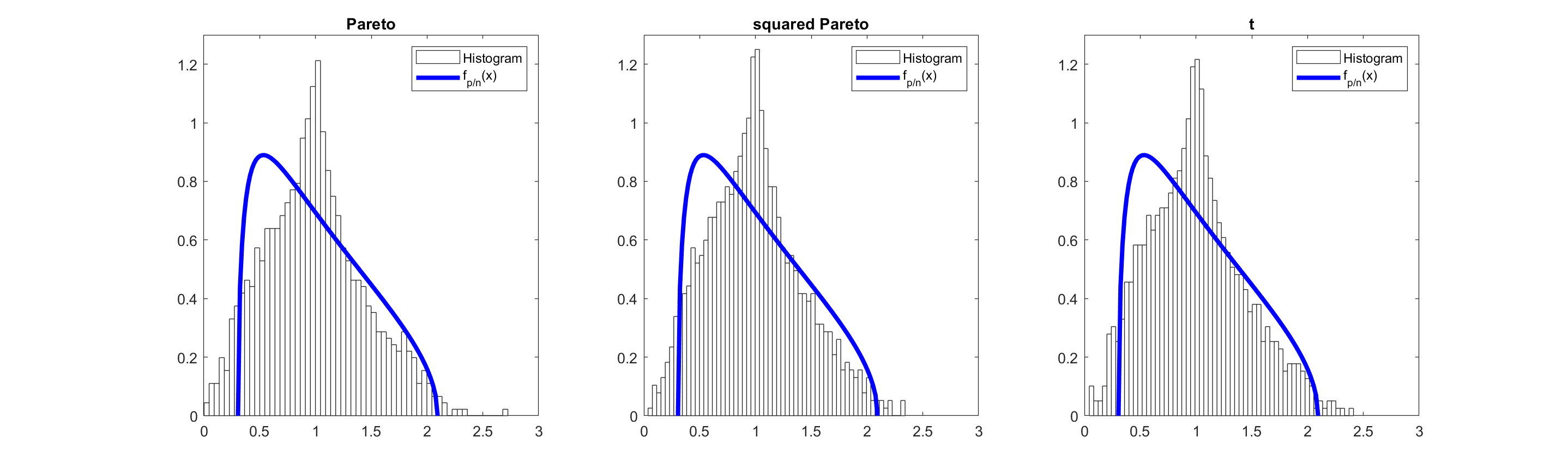}
\caption{Histograms of the \ahmplaw~with parameter $p/n$ and $\alpha=1.1$. The data are
  simulated from the following distributions: (left) $\xi\eid X-\E X$ with $X\sim \text{Pareto}(\alpha)$, (middle) $\xi\eid X^2-\E X^2$ with $X\sim \text{Pareto}(2\alpha)$, (right) $\xi$ $t$-distributed with $\alpha$ degrees of freedom. The dimension is $p=1000$ and the sample size  $n=5000$.}
\label{fig:nonsymmetric}
\end{figure}

In Figure \ref{fig:nonsymmetric} we consider the spectral distribution of $\bfR$ for non-symmetric distributions of $\xi$. The left and middle plots in Figure \ref{fig:nonsymmetric} are generated with non-symmetrically distributed $\xi$'s that are regularly varying with index $\alpha=1.1$. For the right plot, the symmetric $t$-distribution with $\alpha$ degrees of freedom was used for which Theorem~\ref{thm:mainsimplified} shows the convergence of the ESD to the \ahmplaw. We observe that that the three ESDs look almost identical, which suggests  that the symmetry assumption in Theorem~\ref{thm:mainsimplified} can indeed be relaxed. 

\begin{figure}[htb!]
  \centering
    \includegraphics[trim = 0.75in 2.9in 0.75in 2.7in, clip, scale=0.75]{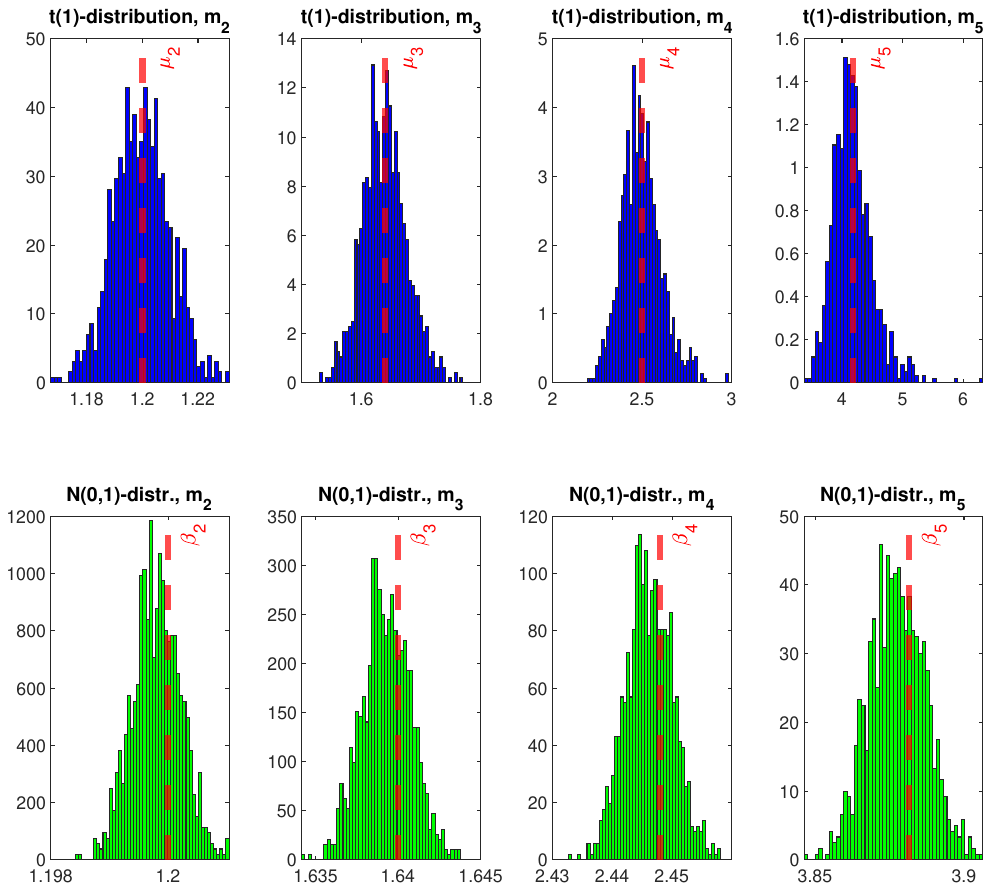}
\caption{Histograms of empirical moments of $m_2,\dots,m_5$ from
  $p\times n$ data matrices with $t(1)$-distributed entries (top row)
  and $N(0,1)$-distributed entries (bottom row); $p=1000$,
  $n=5000$ and using $L=1000$ independent replicates.}
\label{fig:moments1}
\end{figure}

Theorem~\ref{thm:mainsimplified} also yields the limits of the
empirical moments $m_k(\bfR):= p^{-1} \sum_{i=1}^p
(\lambda_i(\bfR))^k$ for $k\ge 1$. More precisely, in the course of
its proof we will show that $m_k(\bfR)$ converges in probability to
$\mu_k(\alpha,\gamma)$. Note that the case $k=1$ is trivial since
$m_k(\bfR)=1$. In Figure~\ref{fig:moments1}, we place ourselves in the
setting of the top right panel of Figure \ref{fig:fouralphas} ($\alpha=1$). That is, we pick $p=1000,n=5000$ and simulate the iid entries of $\bfX^{(j)}$ from a $t(1)$ distribution. Then we compute the eigenvalues $\lambda_1(\bfR^{(j)}), \ldots,\lambda_p(\bfR^{(j)})$. This procedure is repeated until we have $L=1000$ samples 
$$(\lambda_1(\bfR^{(j)}), \ldots,\lambda_p(\bfR^{(j)}))\,, \quad 1\le j\le L\,,$$
from which we calculate $m_k(\bfR^{(j)})$, $1\le j\le L$ for
$k\in\{2,3,4,5\}$. The first row in Figure~\ref{fig:moments1} shows
the (normalized)  histograms of $m_k(\bfR^{(j)})$, ($1\le j\le L$). By
Theorem \ref{thm:mainsimplified}, the limits (in probability)  of $m_k(\bfR^{(j)})$ are 
$$(\mu_2,\mu_3,\mu_4,\mu_5)=(1.2,1.64,2.4980,4.1816)\,,$$
where $\mu_k$ is a shorthand notation for $\mu_k(1,0.2)$. 
Vertical lines at values $\mu_k$ were added to the histograms in the
first row. The averaged empirical moments 
$$\frac{1}{L} \Big(\sum_{j=1}^L m_k(\bfR^{(j)}); k=2,\ldots, 5\Big)= (1.1996,1.6389,2.4956,4.1774),$$
are very  close to their limits $(\mu_2,\mu_3,\mu_4,\mu_5)$.

To obtain the second row in Figure~\ref{fig:moments1}, we simulated from a standard normal distribution instead of the $t(1)$ distribution. In this case the theoretical limiting moments are the \MP moments $\beta_k:=\beta_k(0.2)$, 
$$(\beta_2,\beta_3,\beta_4,\beta_5)=(1.2,1.64,2.448,3.8816)$$
and the averaged empirical moments are  $(1.1998,1.6393,2.4462,3.8776)$. It is interesting to note the different scaling on the $x$-axis when comparing the first and the second row of plots in Figure~\ref{fig:moments1}. In case of normal data, the spread is much smaller than for the heavy-tailed $t$-distribution. For $k\in\{4,5\}$ the $m_k$ fluctuate around different means since $d_k(1,0.2)>0$; see \eqref{eq:dec}.

Our next result shows that the family of
  \ahmplaw s $\{H_{\alpha,\gamma}\}$ can be continuously extended at
its boundaries $\alpha\in\{0,2\}$.

\begin{theorem}\label{thm:limit}
  \begin{enumerate}
  \item[(1)]
    The limit $\lim_{\alpha \to    0^+}H_{\alpha,\gamma}$ is a modified Poisson distribution
    with probability mass function $q_\gamma$ 
    \begin{equation}\label{eq:modpoisson}
       q_\gamma(0) = 1-\frac1\gamma +\frac1\gamma \e^{-\gamma}\qquad \text{ and } \qquad
      q_\gamma(k) =\frac{1}{\gamma}\e^{-\gamma}\frac{\gamma^{k}}{k!}\,,\qquad k\ge 1\,.
    \end{equation}
  \item[(2)] The limit 
    $\lim_{\alpha \to 2^-}H_{\alpha,\gamma}$ is the \MP law
    $\sigma_{MP,\gamma}$.
  \end{enumerate}
\end{theorem}

The proof of Theorem~\ref{thm:limit} is given in Section~\ref{sec:boundary}. Theorem \ref{thm:limit} shows that $H_{\alpha,\gamma}$ interpolates between the (modified) Poisson and the \MP distribution which are the boundary cases for $\alpha \to 0^+$ and  $\alpha \to 2^-$.

 Note that $ 1-\frac1\gamma +\frac1\gamma e^{-\gamma}>0$ for $\gamma>0$.
 Compared with the Poisson distribution with parameter
 $\gamma$,  
 the modified Poisson  distribution  $q_\gamma$
 has the masses at $k\ge 1$ scaled by the factor $1/\gamma$, a
 magnification when $\gamma<1$ and a shrinkage otherwise. 
 It has mean $1$ which is very natural. Indeed, the
 $H_{\alpha,\gamma}$ distributions all have mean $1$.
 In particular when $\gamma\to 0$,
 $q_\gamma$ degenerates to the Dirac mass at 1.

Figure \ref{fig:modpoisson} shows normalized histograms of the
spectrum of $\bfR$ for various values of $p,n$ and $\alpha=0.05$. The
plots nicely illustrate the convergence to the modified Poisson
distribution. In the top left  panel, the bars at 0 and 1 are of about
the same height. This is in perfect agreement with the point masses of
the Poisson$(1)$ distribution at 0 and 1 which both are $\e^{-1}$. We also see
that the smaller the ratio  $p/n$, higher the concentration of the
eigenvalues around the mean~1. 

\begin{figure}[htb!]
  \centering
    \includegraphics[trim = 0.75in 2.6in 0.65in 2.5in, clip, scale=0.75]{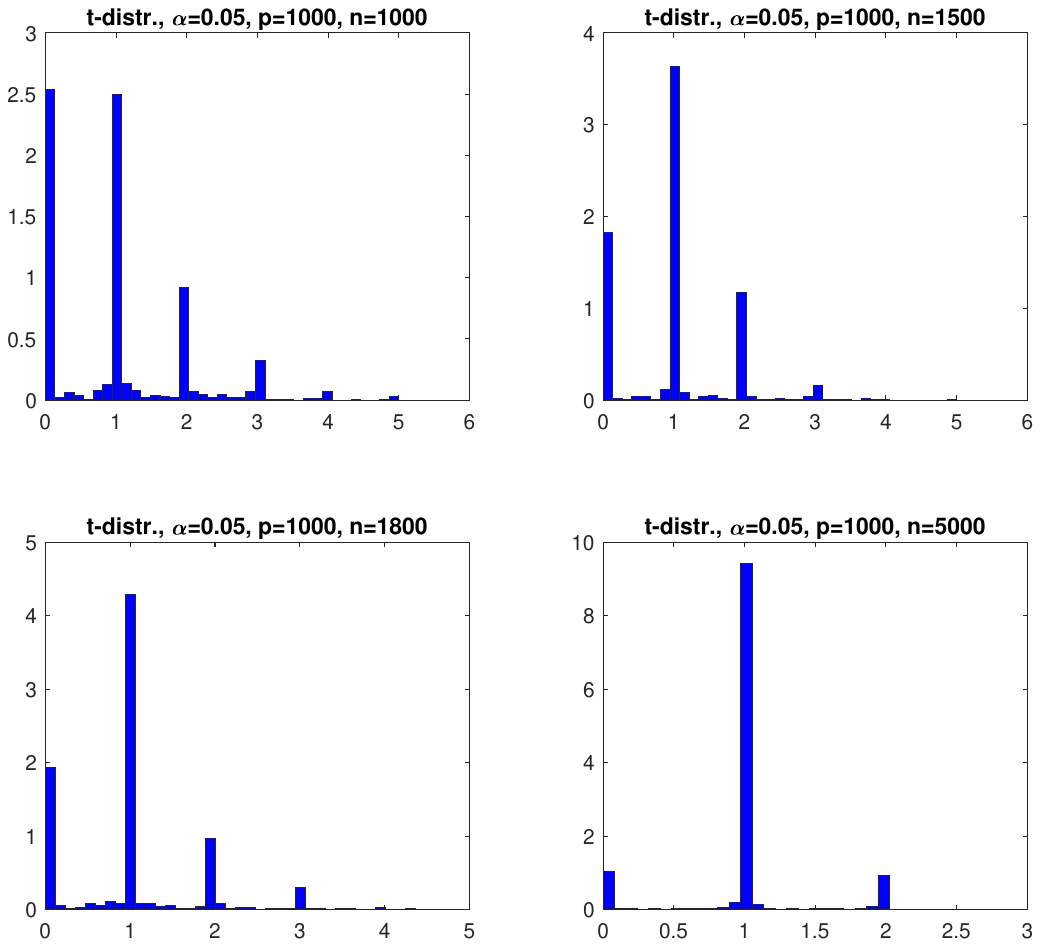}
\caption{Histograms of the \ahmplaw~with parameter $p/n$. The data are
  simulated from a $t$-distribution with $\alpha=0.05$ degree of
  freedom and the ratio $\gamma = p/n$ varying from $1$ to
    $2/3$, $5/9$ and $1/5$, illustrating
    different  limiting modified Poisson distributions $q_\gamma$.}
  \label{fig:modpoisson}
\end{figure}

  \begin{remark}
    Simulation results in Figure~\ref{fig:modpoisson} suggest
    that the
    $\alpha$-heavy MP law might be a mixture of an absolutely
    continuous component and a discrete distribution supported on the
    integers when $0<\alpha<2$.
    It would be interesting  to confirm this point
    rigourously.  However, a study on the support of limiting distributions
    is generally possible if we can characterize the distribution through its
    Stieltjes transform, see \cite{arous:guionnet:2008} and
    \cite{belinschi:dembo:guionnet:2009} for the case of Wigner matrix
    and sample covariance matrix with heavy-tailed entries.
    As this paper is based on the moment method, it is relevant to develop a parallel study using the resolvent method in the future for answering further questions about characteristics of the \ahmplaw.
\end{remark}

\subsection{A statistical application}\label{ssec:estimate}

In this section, we develop a simple application of our general result
to the problem of estimation of the tail index parameter $\alpha$.
As discussed earlier,
the empirical moments $m_k(\bfR):= p^{-1} \sum_{i=1}^p
(\lambda_i(\bfR))^k$ for $k\ge 1$
converge in probability to the corresponding  moments  
$\mu_k(\alpha,\gamma)$ of the limiting $\alpha$-heavy MP law.
Note that  for $k\ge  4$, $\mu_k(\alpha,\gamma)$ is a bijective function 
of $\alpha$ for any fixed ratio parameter $\gamma$.
Let $\gamma_n=p/n$ be the actual dimension ratio. Define 
$\hat\alpha_n$, 
the method of moment estimator
of $\alpha$, as  the solution of  
\begin{equation}\label{eq:alphahat}
  \mu_k(\alpha, \gamma_n) = m_k(\bfR).
\end{equation}
The previous discussion readily yields the consistency of this moment
estimator. 

\begin{proposition}\label{prop:consistency}
Under the conditions of Theorem~\ref{thm:mainsimplified}, for
each $k\ge 4$, the moment estimator  $\hat\alpha_n$ converges in
probability to $\alpha$ as $n,p\to\infty$.
\end{proposition}

For illustration purpose, we consider hereafter the case of $k=5$ with
$\mu_5(\alpha,\gamma)$ as in \eqref{eq:dec} and \eqref{eq:d45}.
Then the estimator is explicitly given by
\[     \hat\alpha_n= 2\left[
  1 - \left(  \frac{ | m_5(\bfR)- \beta_5(\gamma_n)|}{  5\gamma_n(1+\gamma_n)}\right)^{\frac12}\right],
\]
where $\beta_5(\gamma)$ is the fifth moment of the MP law (see \eqref{eq:momentsmp}).

A small simulation experiment is  conducted to check the finite-sample
performance of the estimator. The experimental design is as follows.
\begin{itemize}
\item Independent entries are simulated with Student $t(\alpha)$
  distributions with $\alpha\in\{0.5,1,1.5\}$;
\item The dimension ratio is fixed to $n=2p$ with 
  $p$ varying  from 100 to 800 (thus $n$ runs from 200 to 1600);
\item  For each combination, the estimator $\hat\alpha_n$ is averaged
  over 1000 independent replications.
\end{itemize}

These estimates are reported  in  Table~\ref{tab:estimates} and
plotted in Figure~\ref{fig:estimates}. One can observe the consistency
of the estimator. However, it is also noticed that this convergence is
in general slow which is indeed expected for heavy tailed entries.
\begin{table}[hp!]
  \caption{Empirical averages of the
    estimator $\hat\alpha$ for various dimension combinations.
    Entries with 
    independent $t(\alpha)$ variables with true
    $\alpha\in\{0.5,1,1.5\}$.
    The     dimension 
    $p$ varies  from 100 to
    800 and $n=2p$.  Each averaged estimate 
    of $\hat\alpha$ uses 1000 replications.\label{tab:estimates}}
  \centering
  \begin{tabular}{c|rrrrrrrr}
  $p$ &   100&      200&         300&      400&      500&        600&    700&       800\\  \hline
  $\alpha=0.5$ &  0.227  & 0.448 & 0.482 & 0.537 & 0.597 & 0.563 & 0.579 & 0.601\\
  $\alpha=1.0$   &  0.577 & 0.773 & 0.896 &  0.901&  0.931& 0.955& 1.011&    1.006\\ 
  $\alpha=1.5$ &  0.943 & 1.103 & 1.203 & 1.237 & 1.297 & 1.326 & 1.328 & 1.360\\
    \hline
  \end{tabular}
\end{table}

\begin{figure}[tbhp!]
  \centering
  \includegraphics[width=0.3\linewidth]{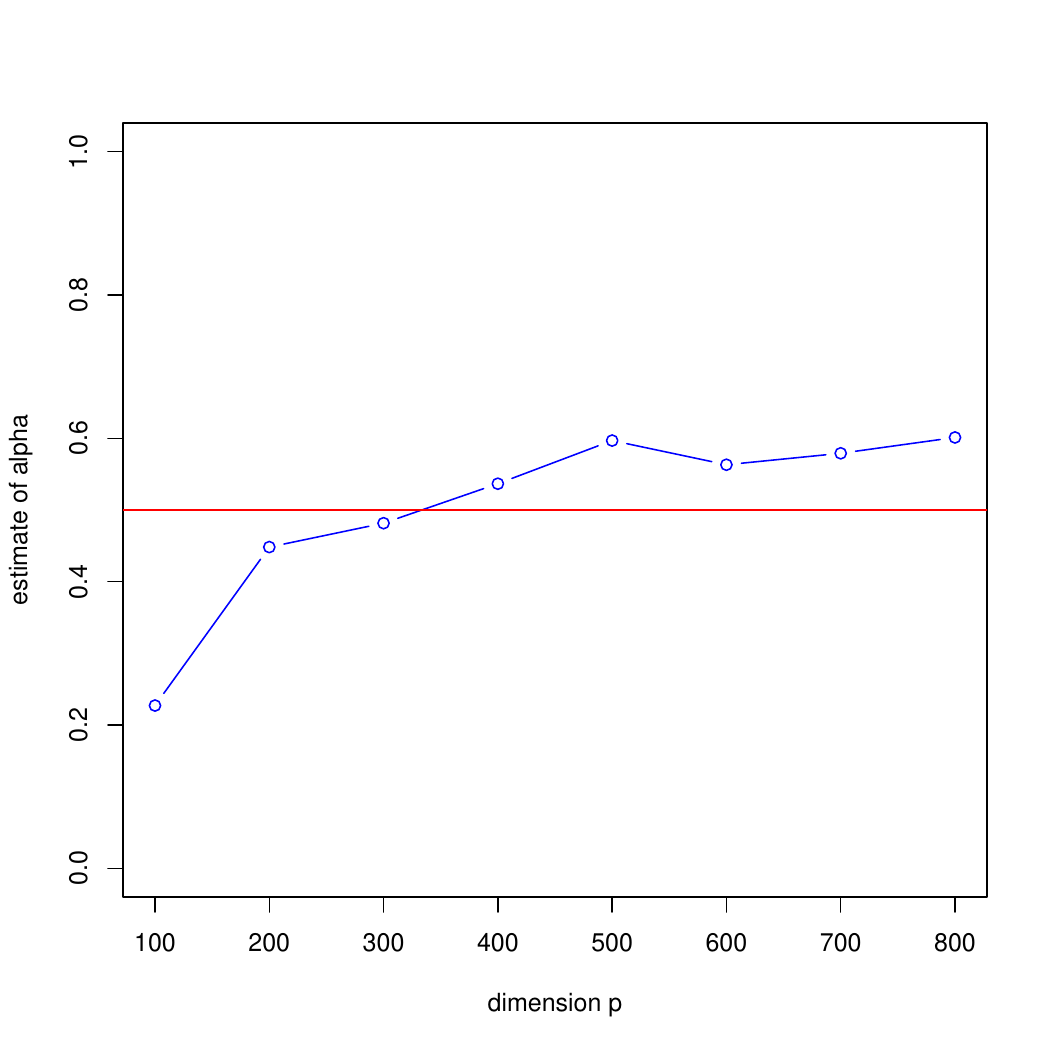}
  ~
  \includegraphics[width=0.3\linewidth]{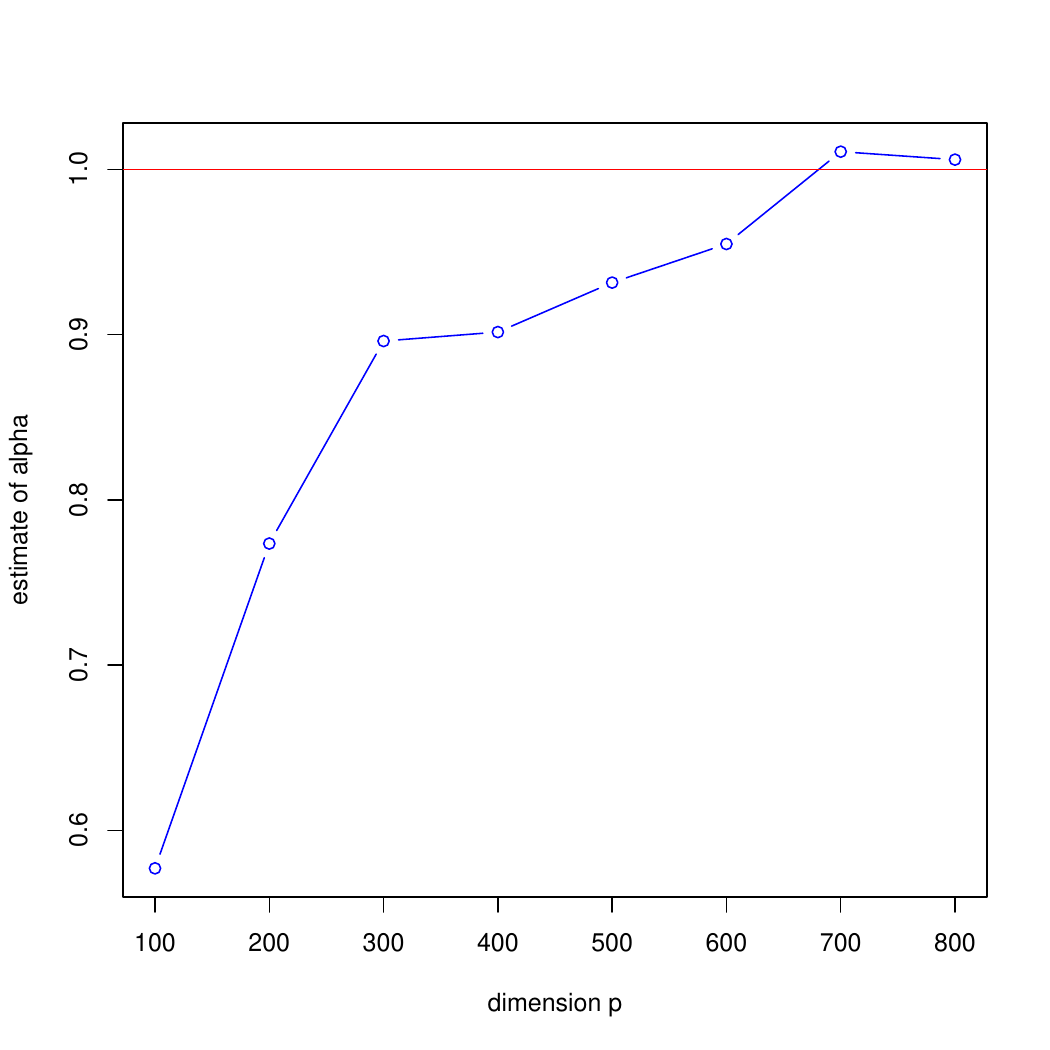}
  ~
  \includegraphics[width=0.3\linewidth]{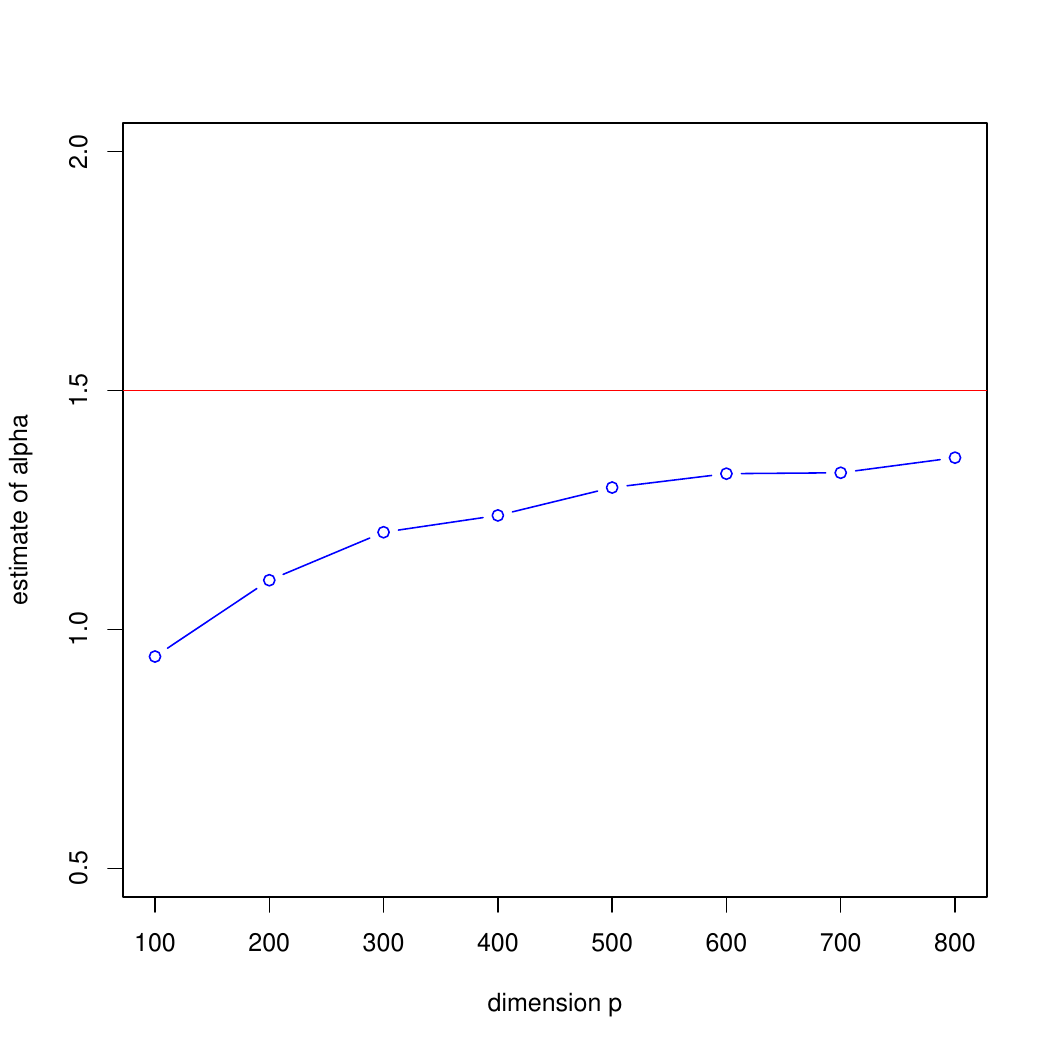}
  
  \caption{
    Evolution of
    the estimator $\hat\alpha$ when dimensions grow.  Entries $X_{ij}$ are 
    independent $t(\alpha)$ variables with
    true $\alpha=0.5$ (left
    panel),  $\alpha=1$ (middle  panel)  and $\alpha=1.5$ (right
    panel). The dimension 
    $p$ varies  from 100 to
    800 and $n=2p$.  Each averaged estimate 
    of $\hat\alpha$ uses 1000 replications.\label{fig:estimates}
  }
\end{figure}

\section{Method of moments for $\bfR$}\label{sec:moments}

In the following sections, we are interested in the $k$-th moment of the ESD $F_{\bfR}$ of the sample correlation matrix $\bfR$ given by
\begin{equation}
  \label{eq:mk}
  m_k = \int x^k d F_\bfR (x) = \frac{ \tr \bfR^k}{p} =\frac1p
  \sum_{i_1,\ldots,i_k=1}^p  \sum_{t_1,\ldots,t_k=1}^n
  Y_{i_1t_1}  Y_{i_2t_1} Y_{i_2t_2}  
	\cdots  Y_{i_kt_k}    Y_{i_{k+1}t_k}.
\end{equation}
(Here the convention $i_{k+1}=i_1$ is used.) Throughout $(X_{it})$ are iid symmetric, which implies that the $Y_{it}$ are symmetric as well. 

So far the method of moments which is one of the main techniques in
random matrix theory has not  been applied to sample correlation
matrices (up to our best knowledge).

  The  reason might be  that the normalization in the variables
  $\{Y_{it}\}$ implies that $Y_{it}$ and $Y_{it'}$ are dependent for
  $t\ne t'$ while the similar quantities are independent in the case
  of a sample covariance matrix.  This difference makes the moment
  calculation more involved as seen  in the subsequent developments in
  this section.

\subsection*{Outline of the proof of Theorem~\ref{thm:mainsimplified}}

In order to compute the limit of $\E[m_k]$, we perform a series of simplifying steps that innovatively use the inherent structure of self-normalized random variables and combine them with a moment formula derived from \cite{albrecher:teugels:2007} for the product of the $Y$'s on the \rhs~of \eqref{eq:mk}. In Section~\ref{sec:momentsofR}, we define the function $$F(I)=\sum_{t_1,\ldots,t_k=1}^n
 \E\big[ Y_{i_1t_1}  Y_{i_2t_1}  \cdots  Y_{i_kt_k}    Y_{i_{k+1}t_k} \big]$$ for a so-called path $I=(i_1,\ldots,i_k)$ and study properties of $F(I)$. We employ two reduction steps which mainly use the fact that $Y_{11}^2+\cdots+Y_{1n}^2=1$ to simplify $F(I)$. More precisely, we transform the path $I$ to a shorter path $\wt I$ in such a way that $F(I)$ can be obtained from $F(\wt I)$. Depending on their reducibility, we distinguish between three classes of paths: completely reducible paths,
irreducible paths and partially reducible paths. It is noteworthy that the completely reducible paths eventually yield the $k$-th \MP moment $\beta_k(\gamma)$. Our path-shortening procedure is applicable to sample correlation matrices as long as the distribution of $\xi$ is symmetric. Under finite variance, one may deduce from Remark~\ref{rem:3.2} that $\E[m_k]$ tends to $\beta_k(\gamma)$ since only the completely reducible paths contribute to the limit. The main challenge in the infinite variance case, where $\xi$ is regularly varying with index $\alpha\in (0,2)$, lies in the fact that all paths $I$ have a non-negligible contribution (see \eqref{eq:see}). By \eqref{moment} and the path-shortening procedure, the contribution of $F(I)$ to the sum in \eqref{eq:mk} depends only on the shortened path $\wt I$ and $\alpha$. Therefore, we need some rather involved combinatorics to accurately count how many $I$'s reduce to the same $\wt I$ under the path-shortening procedure. To this end, the first part of Section~\ref{sec:F(I)} establishes the combinatorics on associated graph counting by using set partitions. The remainder of Section \ref{sec:F(I)} is devoted to the calculation of $F(I)$, ultimately culminating in Proposition \ref{prop:calcF} which is the crucial auxiliary result in the proof of Theorem~\ref{thm:mainsimplified}. All our path-shortening steps in Section \ref{sec:moments} and combinatorics in Section \ref{sec:F(I)} are illustrated by additional examples for the reader's convenience. The proof of Theorem~\ref{thm:mainsimplified} is then completed in Section \ref{sec:mainresult}.

\subsection*{Self-normalized moments}
To compute the expectation of $m_k$, we need to understand the even
moments of products of self-normalized $Y_{ij}$'s. Assuming that
$\xi$ is regularly varying
with index $\alpha<2$, Albrecher and Teugels
\cite[p.~4]{albrecher:teugels:2007} derived the following formula for
the moments of the self-normalized random variables
\begin{equation}\label{moment}
\binom{n}{r} \E[Y_{11}^{2k_1} Y_{12}^{2k_2}\cdots Y_{1r}^{2k_r}] \sim \frac{\Big(\frac{\alpha}{2}\Big)^{r-1} \prod_{j=1}^r \Gamma(k_j-\alpha/2)}{r \Big(\Gamma(1-\alpha/2)\Big)^r \Gamma(k)}\,,\qquad \nto,
\end{equation}
where $k\ge 1$, $1\le r\le k$,  $k_i\ge 1$ and $k_1+\cdots+k_r=k$; and $\Gamma(\cdot)$ denoting the gamma function. In particular, we have
\begin{equation}\label{highestmoment}
n \E[Y_{11}^{2k} ] \sim \frac{\Gamma(k-\alpha/2)}{\Gamma(1-\alpha/2) \Gamma(k)}\,,\qquad \nto.
\end{equation}

\begin{remark}
To be precise, formula \eqref{moment} was stated in \cite[p.~4]{albrecher:teugels:2007} for $k\ge 2$, $r\ge 1$,  $k_i\ge 2$, $k_i$ even, and $k_1+\cdots+k_r=k$. The same proof shows that our more general formulation is valid.
\end{remark}

 It is interesting to compare the values in \eqref{eq:limita<2}
  and \eqref{moment}
  with their counterparts from a Gaussian random variable $\xi\sim
  N(0,1)$. In this case the vector $(Y_{11},\ldots,Y_{1n})$  has the Haar
  distribution on the  unit sphere ${\mathcal{S}}_{n-1}$.  It is well-known
  that  $\E[Y_{11}^4] = 3/(n(n+2))$;
  thus $\lim_{\nto} n\,\E[Y_{11}^4] =0$.
  Moreover, by \cite[Example 2.1]{heiny:mikosch:2017:corr} we have
  \[  \E[Y_{11}^{2k_1} Y_{12}^{2k_2}\cdots Y_{1r}^{2k_r}]  = \frac{\Gamma(n/2) }{2^{k_1+\cdots+k_r} \Gamma(n/2+k_1+\cdots+k_r)} \prod_{j=1}^r (2k_j-1)!!\,.
  \]
Note that $\E[Y_{11}^{2k_1} Y_{12}^{2k_2}\cdots Y_{1r}^{2k_r}]$ is of
order $n^{-(k_1+\cdots+k_r)}$. Unless all $k_i$'s are $1$, this is
much smaller than what we obtained in \eqref{moment} for
$\xi$ regularly varying
 with index $\alpha<2$, where the same expectation was of order $n^{-r}$.

\subsection{Empirical spectral moments of $\bfR$}\label{sec:momentsofR}

In this subsection, we will revisit the path-shortening algorithm developed
in \cite{heiny:mikosch:2017:corr}.
Some terminology from graph theory and notation  is useful. The set of
the first $m$ positive integers is denoted by $\iv{m}$.
A tuple $I=(i_1, i_2, \ldots, i_k)\in {\iv{p}}^k$ of positive integers  is a {\em path} with
vertices $i_\ell \in \iv{p}$.  Its length is $k=|I|$.  
The set of distinct elements in $I$ is denoted by $\{I\}$. For any set $A$, we denote its cardinality by $\# A$. If $r=\#\{I\}$,
$I$ is called an $r$-path.  For example, $I=(1,1,2,2)$ has length 4; it
is a 2-path since $\{I\}=\{1,2\}$. A path is {\em canonical} if $i_1=1$ and $i_l \le \max \{i_1, \ldots,
i_{l-1}\} +1$, $l\ge 2$.  A canonical $r$-path $I$ satisfies $\{I\} =\iv{r}$.

Two paths are {\em isomorphic} if one becomes the other by a suitable
permutation on $\iv{p}$. For example,  $(9,6,9,6)$ and
$(1,2,1,2)$ are isomorphic, but only the latter is canonical. Each
{\em isomorphism class} contains exactly one canonical path.
Given a canonical path with vertices $\iv{r}$ (a $r$-path), its
isomorphic class of paths in ${\iv{p}}^k$ has exactly
$p(p-1)\cdots(p-r+1)$ distinct elements: this corresponds to the
number of injective maps from $\iv{p}$ to $\iv{r}$. Let   $\mathcal{J}_{r,k}(p)$ denote the set of all $r$-paths $I\in \iv{p}^k$.
We then have the disjoint union
\begin{equation}\label{eq:disjointkds}
\iv{p}^k = \bigcup_{r=1}^k \mathcal{J}_{r,k}(p)\,.
\end{equation}
For the reason just explained, it holds that
\[ \#\mathcal{J}_{r,k}(p) = p(p-1)\cdots (p-r+1) \# \mathcal{C}_{r,k}~,
\]
where
\begin{equation}\label{crk}
  \mathcal{C}_{r,k}=\{ \text{canonical $r$-paths of length $k$} \},\, \qquad 1\le r\le k\,.
\end{equation}
For more details and examples of these path notions consult Section 2.1.2 in \cite{bai:silverstein:2010}.

Each summand in \eqref{eq:mk} corresponds to a path 
$I=(i_1, i_2, \ldots, i_k)$ with vertices in $ \iv{p}$ 
and a path
$T=(t_1, i_2, \ldots, t_k)$ with vertices in $ \iv{n}$. 
Let 
\begin{eqnarray}
  \label{eq:fI}
  f(I,T) &=& f_n(I,T) = \E\left[
  Y_{i_1t_1}  Y_{i_2t_1} Y_{i_2t_2}  Y_{i_3t_2} Y_{i_3t_3} \cdots  Y_{i_kt_k}    Y_{i_{k+1}t_k}
             \right],  \\
  \label{eq:fctF}
  F(I) & =  &  F_n(I)= \sum_{T \in \iv{n}^{|I|}} f(I,T)\,,
\end{eqnarray}
where the dependence on $n$ is suppressed in our notation. Note that on the \rhs~of \eqref{eq:fctF} we have $|I|=k$; but we prefer to write $|I|$ to indicate that $F$ can be applied to paths of any length. By convention, $\emptyset$ denotes the empty path and we set
$F(\emptyset)=n$.

By \eqref{eq:mk}, we have then
\[  \E[m_k] =  \frac1p \sum_{I \in {\iv{p}}^k} \sum_{T \in \iv{n}^k} f(I,T)= \frac1p \sum_{I \in {\iv{p}}^k} F(I)\, .
\] 
We rewrite $\E[m_k]$ by sorting according to the number of distinct
components in the path $I$.
Note that $F(I_1)=F(I_2)$ if $I_1$ and $I_2$ are isomorphic. In view of \eqref{eq:disjointkds}, we see that
\begin{equation}\label{eq:okljjj}
\begin{split}
\E[m_k]&= p^{-1}\sum_{r=1}^k \sum_{I\in \mathcal{J}_{r,k}(p)} F(I)
= p^{-1}\sum_{r=1}^k p(p-1)\cdots (p-r+1)  \sum_{I\in \mathcal{C}_{r,k}} F(I)\\
&\sim \sum_{r=1}^k  \sum_{I\in \mathcal{C}_{r,k}} p^{r-1} F(I)\,, \qquad \nto\,.
\end{split}
\end{equation}
Therefore, the main task is 
to determine the function $F(I)$ for $I\in \mathcal{C}_{r,k}$.
Assume that the symmetrically distributed
 $\xi$ is regularly varying  with index $\alpha<2$.
Using the moment formula \eqref{moment} it is easy to see that every $I\in \mathcal{C}_{r,k}$ has a non-negligible contribution to the limit of $\E[m_k]$. Indeed, since $p$ and $n$ are proportional, we have for $I\in \mathcal{C}_{r,k}$,
\begin{equation}\label{eq:see}
p^{r-1} F(I) = \sum_{T \in \iv{n}^{k}} p^{r-1} f(I,T) 
\ge n p^{r-1} f(I, (1,\ldots, 1))=
n p^{r-1} \prod_{i=1}^r \E\Big[   Y_{i1}^{2N_i} \Big]=O(1)\,,
\end{equation}
where $N_i\ge 1$ counts the number of occurrences of the integer $i$ in $I$; compare also \eqref{eq:fI2} later on. Different paths $I_1,I_2\in \mathcal{C}_{r,k}$ will in general lead to different limits of $p^{r-1} F(I_1)$ and $p^{r-1} F(I_2)$.

\begin{remark}\label{rem:3.2}
If the symmetrically distributed $\xi$ is more light-tailed in the sense
that $n\E[Y_{11}^4] \to 0$, then the values of $F$ can be calculated
more easily. It follows from \cite{heiny:mikosch:2017:corr} or our
path-shortening arguments in Section~\ref{sec:ps} that
\begin{eqnarray}\label{eq:newD}
 F(I)=
\left\{\begin{array}{ll}
n^{1-r} \,, & \mbox{if } I\in \mathcal{C}_{r,k}^0, \\
n^{1-r} O(n\E[Y_{11}^4])  \,, & \mbox{otherwise},
\end{array}\right. \qquad I\in \mathcal{C}_{r,k}\,.
\end{eqnarray}\noindent
Moreover, the cardinality of $\mathcal{C}_{r,k}^0$ (defined in \eqref{eq:C0klj}) is well known (see \eqref{lem:lemma3.4}) which immediately yields that $\lim_{\nto} \E[m_k]=\beta_k(\gamma)$.
\end{remark}

In what follows, we will present several simplifications that can be applied in the calculation of $F(I)$.
\begin{itemize}
\item In Section~\ref{sec:ps}, we introduce a function $S$ which transforms a path $I$ into a certain path $S(I)$ of shorter length. A simple relation between $F(I)$ and $F(S(I))$ is derived. 
\item In Section \ref{sec:application}, this relation is applied to simplify the computation of $\E[m_k]$.
\end{itemize}

\subsection{Preliminary reduction by path-shortening}\label{sec:ps}

In this paper we  will heavily consider a class of so-called $\Delta$-graphs defined
as follows \cite[Section 2.1.2]{bai:silverstein:2010}. Let $(I,T)$ be  a pair of paths of length $k$ with vertices in $\iv{p}$
and $\iv{n}$, respectively. Plot the $i_\ell$ vertices and $t_\ell$
vertices on two parallel lines.   For each $1\le \ell\le k$, draw a
down-edge from $i_\ell$ to $t_\ell$, and an up-edge from $t_\ell$ to
$i_{\ell+1}$. This is the $\Delta$-graph associated to the pair
$(I,T)$, denoted as $\Delta(I,T)$. An example of such a graph with $k=3$
is given in Figure~\ref{fig:delta}.

\begin{figure}[htp]
  \centering
  \includegraphics[width=0.6\linewidth]{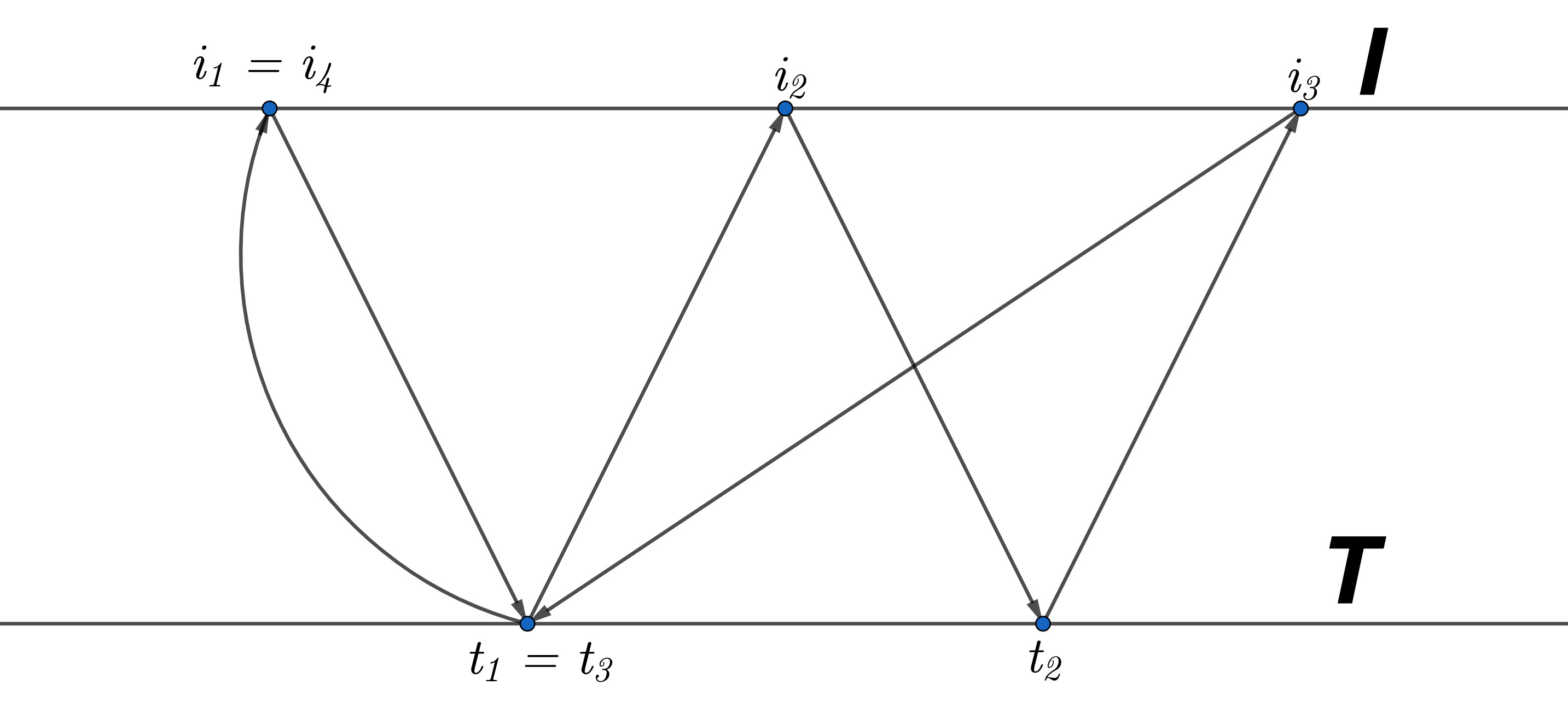}
  \caption{Graph $\Delta(I,T)$  with $I=(i_1,i_2,i_3)$  and $T=(t_1,t_2,t_1)$.
    \label{fig:delta}}
\end{figure}

Furthermore, we remove the orientation of edges: an edge $e=(i,t)\in
\Delta(I,T)$ means either a down-edge $i{\small\searrow} t$ or an up-edge
$t{\small \nearrow} i$.
Recall that the rows of $\Y=(Y_{it})$ are independent.
Using  $\Delta$-graphs, the  product in \eqref{eq:fI} can be expressed as
\begin{equation}
  \label{eq:fI2}
  f(I,T) = \E \left[ \prod_{e=(i,t) \in \Delta(I,T)}  Y_{it}  \right]
   =  \prod_{i\in \{I\}} \E\Big[ \prod_{t\in \{T\} : e=(i,t) \in \Delta(I,T)}  Y_{it}^{m_{it}(I,T)} \Big] \,,
 \end{equation}
 where $m_{it}(I,T)$ is the degree of the edge $(i,t)$ in the
 graph $\Delta(I,T)$.

For future considerations we define  the so-called skeleton
$\Delta^0(I,T)$ of the graph $\Delta(I,T)$, which  is constructed from
$\Delta(I,T)$ by setting all degrees $m_{it}$ equal to $1$.  In
other words, all multiple edges in $\Delta(I,T)$ are glued
together. By construction, $\Delta^0(I,T)$ is a connected graph with
the same vertices as $\Delta(I,T)$. Finally, we write $N_e(I,T)$ for
the number of edges of $\Delta^0(I,T)$.

The matrix $\Y=(Y_{it})$ possesses the following  properties:
\begin{enumerate}
\item[(1)] By symmetry of the entry distribution we have for $s\le n$ that $\E [Y_{i1}^{m_1}\cdots Y_{is}^{m_s}]=0$ if at least one exponent $m_j\in\N$ is odd.
\item[(2)] $\Y$ has independent rows.
\item[(3)] By definition, $\sum_{t=1}^n Y^2_{it}=1$ for each row $i$.  
\end{enumerate}

Assume that in a  $\Delta(I,T)$ graph, there is an edge $(i,t)$ with
odd degree, say $2s_{it}+1$.
By \eqref{eq:fI2} and property (1), $f(I,T)=0$ and this
graph will not contribute to $F(I)$.
Therefore, in the remaining discussions, we may assume that all
degrees $m_{it}=2s_{it}$ are even.
It follows by \eqref{moment} that 
\begin{equation}\label{eq:orderfit}
  f(I,T) \sim  \theta(I,T) n^{-N_e(I,T)} \,, \qquad \nto\, ,
\end{equation}
for some positive constant $\theta(I,T)$.
Therefore, $f(I,T)>0$ will be of highest order
if $N_e(I,T)$ is minimal.

The other two properties (2) and (3) will allow useful simplifications in the
calculations of $F(I)$.  Typically, a path $I$ will be reduced to a
shorter path $S(I)$.

\medskip\noindent{\bf Type-I reduction: elimination of runs.}\quad
We say that a
{\em run} is formed in $I$ when two consecutive vertices are equal,
that is $i_\ell=i_{\ell+1}$ for some $\ell \in \iv{k}$. For example,
both $I=(1,1,2,2)$ and $I'=(1,1,1,2)$ contain two runs.
Such a run corresponds to a product of the form
\begin{eqnarray*}
  &&  Y_{i_1t_1}  Y_{i_2t_1} \cdots Y_{i_{\ell-1} t_{\ell-1}}   Y_{i_{\ell}t_{\ell-1}}\cdot
     Y_{i_\ell t_\ell}  Y_{i_{\ell+1}t_\ell} \cdot 
     Y_{i_{\ell+1}t_{\ell+1}} Y_{i_{\ell+2}t_{\ell+1}}  \cdots Y_{i_kt_k}    Y_{i_{k+1}t_k}
  \\
  &&  = Y_{i_1t_1}  Y_{i_2t_1} \cdots Y_{i_{\ell-1} t_{\ell-1}}   Y_{i_{\ell}t_{\ell-1}} \cdot
     \boxed { ~Y_{i_\ell t_\ell}^2 ~}   \cdot 
     Y_{i_{\ell+1}t_{\ell+1}} Y_{i_{\ell+2}t_{\ell+1}}  \cdots Y_{i_kt_k}    Y_{i_{k+1}t_k}
\end{eqnarray*}
Therefore, we can isolate the sum over $t_\ell$ of the squares in the
box, and as 
\[  \sum_{t_\ell=1}^n Y_{i_\ell t_\ell}^2=1, 
\]
we obtain 
\begin{eqnarray*}
  \lefteqn{F(I) =  \sum_{T \in \iv{n}^k} f(I,T) } \\
& = &   \sum_{T \in \iv{n}^k}
      \E\left[    Y_{i_1t_1}  Y_{i_2t_1} \cdots Y_{i_{\ell-1} t_{\ell-1}}   Y_{i_{\ell}t_{\ell-1}}\cdot
      Y_{i_\ell t_\ell}  Y_{i_{\ell+1}t_\ell} \cdot 
      Y_{i_{\ell+1}t_{\ell+1}} Y_{i_{\ell+2}t_{\ell+1}}  \cdots Y_{i_kt_k}    Y_{i_{k+1}t_k}  \right]
  \\
&=&
    \sum_{ (t_1,\ldots, t_{\ell-1}, t_{\ell+1},\ldots, t_k)  \in \iv{n}^{k-1}}
   \!\!\!\!\!\!\!\!\!\!\!\!\!\!\!\!\!\! \E\left[    Y_{i_1t_1}  Y_{i_2t_1} \cdots Y_{i_{\ell-1} t_{\ell-1}}   Y_{i_{\ell}t_{\ell-1}}\cdot
    1\cdot Y_{i_{\ell+1}t_{\ell+1}} Y_{i_{\ell+2}t_{\ell+1}}  \cdots Y_{i_kt_k}    Y_{i_{k+1}t_k}  \right]
  \\
&=&
    \sum_{ (t_1,\ldots, t_{\ell-1}, t_{\ell+1},\ldots, t_k)  \in \iv{n}^{k-1}}
   \!\!\!\!\!\!\!\!\!\!\!\!\!\!\! \E\left[    Y_{i_1t_1}  Y_{i_2t_1} \cdots Y_{i_{\ell-1}
    t_{\ell-1}}   Y_{\mathbf{i_{\ell+1}}t_{\ell-1}}\cdot
      Y_{i_{\ell+1}t_{\ell+1}} Y_{i_{\ell+2}t_{\ell+1}}  \cdots Y_{i_kt_k}    Y_{i_{k+1}t_k}  \right]
  \\
&=&
   F(\tilde I),  
\end{eqnarray*}
where the new path $\tilde I=(i_1,\ldots, i_{\ell-1}, i_{\ell+1},\ldots,i_k) $ 
has one vertex less. Naturally, the process can be repeated if the new
path includes further runs.

\medskip\noindent{\bf Type-II  reduction: elimination of simple $i$-vertices.}\quad
Assume that an index $i_\ell$ appears in $I$ exactly once.  We say that $i_\ell$ is simple. The above
product reads as
\begin{eqnarray*}
  &&  Y_{i_1t_1}  Y_{i_2t_1} \cdots Y_{i_{\ell-1} t_{\ell-1}}   Y_{i_{\ell}t_{\ell-1}}\cdot
     Y_{i_\ell t_\ell}  Y_{i_{\ell+1}t_\ell} \cdot 
     Y_{i_{\ell+1}t_{\ell+1}} Y_{i_{\ell+2}t_{\ell+1}}  \cdots Y_{i_kt_k}    Y_{i_{k+1}t_k}
  \\
  &&  = Y_{i_1t_1}  Y_{i_2t_1} \cdots Y_{i_{\ell-1} t_{\ell-1}} \cdot
     \boxed { ~  Y_{i_{\ell}t_{\ell-1}} 
     Y_{i_\ell t_\ell} ~}   \cdot   Y_{i_{\ell+1}t_\ell} 
     Y_{i_{\ell+1}t_{\ell+1}} Y_{i_{\ell+2}t_{\ell+1}}  \cdots Y_{i_kt_k}    Y_{i_{k+1}t_k}.
\end{eqnarray*}
The boxed terms are the only terms with index $i_\ell$; they are
independent of the rest, and their expectation factorizes out with value
\[   \E [ Y_{i_{\ell}t_{\ell-1}}      Y_{i_\ell t_\ell}   ] =
  n^{-1}\1_{\{ t_{\ell-1} =t_{\ell}\}  }~.
\]
Therefore we have
\begin{eqnarray*}
  \lefteqn{F(I) =  \sum_{T \in \iv{n}^k} f(I,T) } \\
& = &   \sum_{T \in \iv{n}^k}
      \E\left[
      Y_{i_1t_1}  Y_{i_2t_1} \cdots Y_{i_{\ell-1} t_{\ell-1}}   Y_{i_{\ell}t_{\ell-1}}\cdot
      Y_{i_\ell t_\ell}  Y_{i_{\ell+1}t_\ell} \cdot 
      Y_{i_{\ell+1}t_{\ell+1}} Y_{i_{\ell+2}t_{\ell+1}}  \cdots Y_{i_kt_k}    Y_{i_{k+1}t_k}
      \right]
  \\
&=&
    \sum_{ T  \in \iv{n}^{k}}
   \!\!\! \E\left[
    Y_{i_1t_1}  Y_{i_2t_1} \cdots Y_{i_{\ell-1} t_{\ell-1}} \cdot
    n^{-1}\1_{\{ t_{\ell-1} =t_{\ell}\}  }
    \cdot   Y_{i_{\ell+1}t_\ell} 
    Y_{i_{\ell+1}t_{\ell+1}} Y_{i_{\ell+2}t_{\ell+1}}  \cdots Y_{i_kt_k}    Y_{i_{k+1}t_k}
    \right]
  \\
&=&
    n^{-1}
   \!\!\!\!\!\! \sum_{ (t_1,\ldots, t_{\ell-1}, t_{\ell+1},\ldots, t_k)  \in \iv{n}^{k-1}}
    \hspace{-1cm} \!\!\!\E\left[
    Y_{i_1t_1}  Y_{i_2t_1} \cdots Y_{i_{\ell-1} t_{\ell-1}} 
    \cdot   Y_{i_{\ell+1}  \mathbf{t}_\mathbf{\ell-1}} 
    Y_{i_{\ell+1}t_{\ell+1}} Y_{i_{\ell+2}t_{\ell+1}}  \cdots Y_{i_kt_k}    Y_{i_{k+1}t_k}
    \right]
  \\
&=&
   n^{-1} F(\tilde I),  
\end{eqnarray*}
where again,  the new path $\tilde I=(i_1,\ldots, i_{\ell-1}, i_{\ell+1},\ldots,i_k) $ 
has one vertex less. Hence, the Type-II reduction removes simple vertices. One can repeat Type-II reductions  if the new
path includes further simple vertices.  Because each  reduction
generates an  $n^{-1}$ factor, it is  important to keep track of the number of Type-II reductions.

\begin{definition}\label{def:pathshortening}
  The process of iterating, whenever possible, the  previous two types
  of reductions on a given path $I$ 
  is referred to as  
  the {\em  Path-Shortening Algorithm (PSA)}.

  The path-shortening function $PS$ applied to a path $I$ is the output
  $(S(I),\runs(I),\simples(I))$ of the algorithm 
  where $S(I)$ is the resulting shortened path, $\runs(I)$ is the total number of vertices that were removed by Type-I reductions and $\simples(I)$ is the total number of vertices that were removed by Type-II reductions.  We write 
\begin{equation*}
PS(I)= (S(I),\runs(I),\simples(I))\,.
\end{equation*}
Finally, a path $I$ is {\em irreducible} if $S(I)=I$.
\end{definition}

$PS(I)$ is the output of the following algorithm.\medskip

\noindent \textbf{Path-Shortening Algorithm $PS(I)$.}
\begin{itemize}
\item[Input: ] Path $I=(i_1,\ldots,i_k)$. Set $J=I$ and $\simples=0, \runs =0$. 
\item[Step 0:] Set $l= |I|$. Go to Step 1.
\item[Step 1:] Erase runs. 
\begin{itemize}
\item If $i_j=i_{j+1}$ for some $1\le j\le l$, where we interpret $i_{l+1}$ as $i_1$, erase element $i_j$ from the path. Set $I=(i_1, \ldots, i_{j-1}, i_{j+1}, \ldots, i_l)$, $\runs = \runs +1$ and return to Step 0.
\item Otherwise proceed with Step 2.
\end{itemize}
\item[Step 2:] Let $s$ be the number of elements of the path $I$ which appear exactly once. Set $\simples:=\simples+s$. Then define $I$ to be the resulting (possibly shorter) path which is obtained by deleting those $s$ elements from the path $I$. Go to Step 3.
\item[Step 3:] 
\begin{itemize}
\item If $J=I$, then return $(I,\runs,\simples)$ as output. 
\item If $J\neq I$, set $J:=I$ and return to Step 0.
\end{itemize}
\end{itemize} 
\medskip

Some simple properties of $PS(I)$ are as follows.
\begin{itemize}
\item For any path $I$, we have the identity 
  \begin{equation*}
    |I|=|S(I)|+\runs(I)+\simples(I)\,.
  \end{equation*}
\item For $I=(1,\ldots,r)$, $S(I)=\emptyset$, which shows that $S(I)$ can have length zero. 
\item By construction, 
  a shortened path $S(I)$ cannot be shortened further: it is
  irreducible.
\item 
  All elements in $S(I)$ appear at least twice. 
\item $|S(I)|$ takes values in the set $\{0,4\}\cup \{6,7,\ldots, |I|\}$. The shortest canonical irreducible path of positive length is $(1,2,1,2)$. 
\item If $I$ is an $r$-path then $\simples(I)= r-1$ is impossible.
\item $\simples(I)$ counts the number of total Type-II reductions
  until no more reduction steps (of Type-I or -II) are possible. Since
  every simple vertex of $I$ can be removed at the very beginning of
  the reduction procedure, it is easy to see that $\simples(I)$ is
  larger or equal to the number of simple vertices in $I$. Indeed, a
  Type-I reduction might create some new  simple vertex in a reduced
  path, thus increasing the number $\simples(I)$.
\end{itemize}

\begin{example}\label{ex:1}
Consider $I=(1,1,2,2)$. Then we have
  \[  F(1,1,2,2) \stackrel{\text{Type I}}{=} F(1,2,2)
    \stackrel{\text{Type I}}{=} F(1,2) \stackrel{\text{Type II}}{=}
    n^{-1}F(1)   \stackrel{\text{Type II}}{=} n^{-2} F(\emptyset) =
    n^{-2}\times n = n^{-1}.  
  \]
	In this case we get $PS(I)=(\emptyset, 2,2)$ and
	the reduction steps directly yield the value of $F$. 

  Next we consider $I=(1,2,1,2,3,3)$. Then we have
  \[  F(1,2,1,2,3,3) \stackrel{\text{Type I}}{=} F(1,2,1,2,3)
    \stackrel{\text{Type II}}{=}
    n^{-1}F(1,2,1,2).
  \]
  Thus the output of the path-shortening algorithm is $PS(I)=((1,2,1,2), 1,1)$.
  The problem of calculating $F(I)$ has been simplified to finding
  $F(1,2,1,2)$ which contains much  fewer terms;  see also \eqref{eq:fctF}.  
\end{example}

The next lemma summarizes the key advantage of path-shortening
for finding values of $F(\cdot)$ (see also \cite[Lemma 4.4]{heiny:mikosch:2017:corr}).

\begin{lemma}\label{prop:PSI}
Assume that the distribution of $\xi$ is symmetric. For any path $I\in \iv{p}^{|I|}$ of finite length, we have
\begin{equation}\label{pathlemma}
F(I)=F(S(I))\, n^{-\simples(I)}\,.
\end{equation}
\end{lemma}

\begin{remark}\label{rem111}
The symmetry requirement on the distribution of $\xi$ is needed for the equality in \eqref{pathlemma}. It allows to neglect all expectations of odd powers of
  matrix entries in our moment method. Without symmetry the \rhs~in \eqref{pathlemma} needs to be multiplied with $(1+o(1))$ as $\nto$. It is possible to modifiy the other arguments of the proof of Theorem~\ref{thm:mainsimplified} accordingly.
	Therefore the symmetry assumption can likely be removed and Theorem \ref{thm:mainsimplified}
 also  holds for non-symmetrically distributed $\xi$. Indeed, this is natural since the moment formula \eqref{moment} which is a key ingredient of the proof only depends on the distribution of $\xi^2$ (and not $\xi$). However, since the current arguments are already involved enough we do not pursue the extension to non-symmetric $\xi$ in this paper.
\end{remark}

\subsection{Application of path-shortening}\label{sec:application}

This subsection explains how path-shortening is used to calculate the $k$-th moment $\E[m_k]$. From \eqref{eq:okljjj} and Lemma~\ref{prop:PSI} we get
\begin{equation}\label{eq:oklj}
\begin{split}
\E[m_k]
&\sim \sum_{r=1}^k   \sum_{I\in \mathcal{C}_{r,k}} p^{r-1} F(I)=  \sum_{r=1}^k   \sum_{I\in \mathcal{C}_{r,k}} p^{r-1} n^{-\simples(I)} \, F(S(I))\\
&= \sum_{r=1}^k   \Big(\sum_{I\in \mathcal{C}_{r,k}^0} + \sum_{I\in \mathcal{C}_{r,k}^1} + \sum_{I\in \mathcal{C}_{r,k}^2} \Big) p^{r-1} n^{-\simples(I)} \, F(S(I))\\
&=: S_{k0}+S_{k1}+S_{k2}\,.
\end{split}
\end{equation}
Here $\mathcal{C}_{r,k}$ is decomposed into the disjoint union $\mathcal{C}_{r,k}^0 \cup \mathcal{C}_{r,k}^1 \cup \mathcal{C}_{r,k}^2$, where  
\begin{equation}\label{eq:C0klj}
\begin{split}
\mathcal{C}_{r,k}^0 &= \{ I\in \mathcal{C}_{r,k}: S(I)=\emptyset \};\\
\mathcal{C}_{r,k}^1 &= \{ I\in \mathcal{C}_{r,k}: S(I)=I \};\\
\mathcal{C}_{r,k}^2 &= \{ I\in \mathcal{C}_{r,k}: 4\le |S(I)|\le k-1 \};
\end{split}
\quad 
\begin{split}
&\text{ completely reducible paths},\\
&\text{ irreducible paths},\\
&\text{ partially reducible paths}.
\end{split}
\end{equation}

First, we shall calculate $S_{k0}$. Lemma 3.4 in \cite{bai:silverstein:2010} determines the cardinality of $\mathcal{C}_{r,k}^0$:
for $k\in \N$ and $1\le r\le k$, 
\beam\label{lem:lemma3.4}
\# \mathcal{C}_{r,k}^0 = \frac{1}{r} \binom{k}{r-1} \binom{k-1}{r-1}\,.
\eeam
For $I\in \mathcal{C}_{r,k}^0$ we have $\simples(I)=r$ and therefore 
\begin{equation*}
F(S(I))\, n^{-\simples(I)} = n^{1-r}\,.
\end{equation*}
In view of $\lim_{\nto} p/n =\gamma$, this implies 
\begin{equation*}
S_{k0}= \sum_{r=1}^k \Big(\frac{p}{n}\Big)^{r-1} \# \mathcal{C}_{r,k}^0 \sim 
\sum_{r=1}^{k} \frac{1}{r} \binom{k}{r-1}\binom{k-1}{r-1}\gamma^{r-1} =\beta_k(\gamma)\,
\end{equation*}
the $k$-th moment of the \MP~law.

\par
Regarding $S_{k2}$, we consider a path $I\in \mathcal{C}_{r,k}^2$. The quantity $\simples(I)$ is easily obtained from the path-shortening algorithm. The shortened path $S(I)$ satisfies $S(S(I))=S(I)$. In words, $S(I)$ is irreducible and hence its canonical representative must be in the set $\mathcal{C}_{r-\simples(I), |S(I)|}^1$. Therefore it suffices to evaluate $F(J)$ for paths $J\in \mathcal{C}_{\widetilde r, \widetilde k}^1$ with $\widetilde r=2,\ldots, r; \widetilde k=4, \ldots, k-1$. 

\begin{remark}
In general, $S(I)$ is not canonical. We prefer to work with canonical
paths which can be nicely described via partitions. In order to
replace $S(I)$ with its canonical representative a simple relabeling
of the vertices is thus  required.
\end{remark}

What is left is to compute $F(I)$ for paths $I\in \mathcal{C}_{r,k}^1$, $r\le k$. This is the content of Section \ref{sec:F(I)} where we also determine the exact size of $\mathcal{C}_{r,k}^1$ which will turn out to be much smaller than $\mathcal{C}_{r,k}$.

\section{Calculation of $F(I)$}\label{sec:F(I)}

 In this section, we present a method to efficiently calculate \eqref{eq:fctF} by identifying those $T$ for which $f(I,T)$ contributes in a non-negligible way. The main theoretical goal is to prove Proposition \ref{prop:calcF}.  As a start, we characterize the sets of possible shortened paths $S(I)$.

\subsection{Precise counting via set partitions}\label{sec:counting}

Let $k\in \N$.
For $1\le r\le k$ an $r$-partition of $\iv{k}$ is a partition of
$\iv{k}$ into exactly $r$ (non empty) sets.
The sets $\mathcal{C}_{r,k}$ and $\mathcal{C}_{r,k}^1$ can be counted via partition numbers. We need the following lemma.
 
\begin{lemma}\label{lem:1to1}
There is a 1-to-1 correspondence between the $r$-partitions of $\iv{k}$ and the canonical $r$-paths of length $k$. 
\end{lemma}
\begin{proof}
Assume that $I=(i_1,i_2,\ldots, i_k) \in \mathcal{C}_{r,k}$. Define the sets
\begin{equation}\label{eq:ergrs}
A_{\ell}=\{j:i_j=\ell \}\,, \quad \ell=1,\ldots,r\,.
\end{equation}
The collection of the sets $A_{\ell}$ forms an $r$-partition of $\iv{k}$. 

Conversely, let $\{B_1, \ldots, B_r\}$ be an $r$-partition of $\iv{k}$. Define the sets $A_{\ell}= B_{v(\ell)}$, where $v(1)$ is such that $1\in B_{v(1)}$ and 
\begin{equation*}
v(\ell) \quad \text{ such that } \min \Big\{ \iv{k} \backslash \Big( \bigcup_{a\in \iv{\ell-1}} A_{a} \Big) \Big\}  \in B_{v(\ell)}\,, \quad \ell=2,\ldots,r\,.
\end{equation*}
Obviously, the sets $(A_{\ell})$ and $(B_{\ell})$ constitute the same partition. Now we obtain a path $(i_1,i_2,\ldots, i_k)$ via
\begin{equation*}
i_j= \arg_{\ell} (j\in A_{\ell}) \,, \quad j=1,\ldots,k\,.
\end{equation*}
It follows easily from this construction that $(i_1,i_2,\ldots, i_k)$ is a canonical $r$-path of length $k$.
The proof is complete.
\end{proof}
In what follows, we will assume without loss of generality that any
sets $B_1,\ldots,B_r$ constituting an $r$-partition of $\iv{k}$ are
listed in the unique order such that $v(\ell)=\ell$, $1\le \ell \le
r$, with the function $\ell\mapsto v(\ell)$ introduced in the above proof. Under this convention, the set $A_{\ell}$ in the $r$-partition constituted by $A_1,\ldots,A_r$ contains the locations of the integer $\ell$ in the path $I$. Conversely, the sets $A_1,\ldots,A_r$ can be recovered from $I$ via \eqref{eq:ergrs}.  

The next result is classical in combinatorics \citep[Chapter V]{Comtet74}.
\begin{lemma}\label{lem:partitions}
The number of $r$-partitions of $\iv{k}$ is the Stirling number of the second kind  given by
\begin{equation}\label{eq:stirling2}
  B(k,r)=\frac{1}{r!}\sum_{j=1}^{r} (-1)^{r-j} \binom{r}{j} j^k.
\end{equation}
The number of partitions of $\iv{k}$ is the $k$-th Bell number $B(k)$, 
\begin{equation*}
B(k)=\sum_{r=1}^k B(k,r)\,.
\end{equation*}
The Bell numbers satisfy the recursion
\begin{equation*}
B(k+1)=\sum_{j=0}^k {k\choose j} B(j).
\end{equation*}
\end{lemma}

A combination of Lemmas \ref{lem:1to1} and \ref{lem:partitions} yields:

\begin{lemma}\label{lem:crk}
The number of canonical $r$-paths of length $k$ is $B(k,r)$, i.e., 
\begin{equation*}
\# \mathcal{C}_{r,k}=B(k,r).
\end{equation*}
\end{lemma}

Next, we count the canonical $r$-paths of length $k$ that remain unchanged by either a Type-I or Type-II reduction. 
We start with Type II, i.e. elimination of simple vertices.

A $2$-associated Stirling number of the second kind is the number of
ways to partition a set of $k$ objects into $r$ subsets, with each
subset containing at least $2$ elements
\citep[page 222]{Comtet74}.  It is denoted by $B_2(k,r)$ and obeys the recurrence relation
\begin{equation*}
  B_2(k+1, r)=r\, B_2(k, r)+k B_2(k-1, r-1)\,.
\end{equation*}
Its generating function is
\[  \sum_{k,r\ge 0} B_2(k,r) u^r \frac{t^k}{k!}=\exp\left\{   u \left( \frac{t^2}{2!}+\frac{t^3}{3!}+\cdots
    \right) \right\}\,.
\]
This leads to 
the closed-form formula
\begin{equation}\label{eq:b2kr}
  B_2(k, r) = \frac{k!}{r!} \sum_{\substack{j_1 + \cdots + j_r=k\\ j_{\ell}\ge 2}} \frac{1}{j_1! \cdots j_r!}\,.
\end{equation}

\begin{lemma}\label{lem:type2}
  There are exactly $B_2(k,r)$ canonical $r$-paths of length $k$
  without any simple vertex; they are thus  invariant under Type-II reductions.
\end{lemma}
\begin{proof}
The $r$-partitions of $\iv{k}$ with each set $A_{\ell}$ having at least $2$ elements are counted by $B_2(k,r)$. The observation that the existence of a simple vertex in a path is equivalent to some set $A_{\ell}$ having just one element finishes the proof.
\end{proof}

Now we count the paths which are invariant under Type-I reductions, i.e. have no runs.

Define the reduced Stirling numbers of the second kind, denoted $B^d(k,r)$, to be the number of ways to partition the integers $\iv{k}$ into $r$ nonempty subsets such that all elements in each subset have pairwise distance at least $d$.  That is, for any integers $i$ and $j$ in a given subset, it is required that $|i-j| \ge d$. It has been shown that these numbers satisfy
\begin{equation*}
B^d(k, r) = B(k-d+1, r-d+1)\,, \quad k \geq r \geq d\,.
\end{equation*}
We will apply this fact with $d=2$. 

\begin{lemma}\label{lem:type1}
There are $\sum_{j=0}^{k-r} (-1)^{j} B^2(k-j,r)=\sum_{j=0}^{k-r} (-1)^{j} B(k-j-1,r-1)$ canonical $r$-paths of length $k$ which are invariant under Type-I reductions. 
\end{lemma}

\begin{proof}
Distance $2$ excludes almost all runs. By our convention $i_1$ and $i_k$ can form a run if they are equal, so we have to take care of them. Hence, there are $B^2(k,r)-D(k,r)$ canonical $r$-paths of length $k$ which are invariant under Type-I reductions. Here $D(k,r)$ denotes the number of ways to partition the integers $\iv{k}$ into $r$ nonempty subsets such that all elements in each subset have pairwise distance at least $2$ and the elements $1$ and $k$ lie in the same set.

It remains to determine $D(k,r)$. In what follows, $\mathcal{P}_k=\{A_1,\ldots, A_r\}$ denotes a partition of the integers $\iv{k}$ into $r$ nonempty subsets such that all elements in each subset have pairwise distance at least $2$. We use the convention $1\in A_1$. For clarification of the notation we remark that the set $A_1$ depends on the partition at hand and might be different from line to line.

 We can obtain each of the $D(k,r)$ partitions above by adding the element $k$ to the set $A_1$ of some $\mathcal{P}_{k-1}$. This works for all $B^2(k-1,r)$ partitions $\mathcal{P}_{k-1}$, except those with $k-1\in A_1$ (because adding $k$ to this set would violate the distance $2$ requirement).  We can create such an exceptional $\mathcal{P}_{k-1}$ by adding $k-1$ to $A_1$ of a partition $\mathcal{P}_{k-2}$. Again this procedure works for all $B^2(k-2,r)$ partitions $\mathcal{P}_{k-2}$, except those with $k-2\in A_1$. We continue until there are no exceptional partitions, i.e. until we reach the partitions $\mathcal{P}_{r}$ because then $r\in A_1$ is impossible since $\mathcal{P}_{r}=\{\{1\},\{2\},\ldots,\{r\}\}$. This shows that 
\begin{equation*}
D(k,r)=B^2(k-1,r)-B^2(k-2,r)+B^2(k-2,r)+\ldots+(-1)^{k-r+1} B^2(r,r)
\end{equation*}
and therefore
\begin{equation*}
B^2(k,r)-D(k,r)=\sum_{j=0}^{k-r} (-1)^{j} B^2(k-j,r)\,.
\end{equation*}
\end{proof}

Our goal is to find the number of canonical $r$-paths of length $k$ which are invariant under both types of reduction. 

\begin{proposition}\label{prop:dfg}
  The number of irreducible canonical $r$-paths of length $k$ is
  \begin{equation*}
    \# \mathcal{C}_{r,k}^1=M(k,r)\,,
  \end{equation*}
where  $M(k,r)$ is the number of $r$-partitions $\{ A_1,\ldots, A_r\}$ of $\iv{k}$ such that:
  \begin{enumerate}
  \item Each $A_{\ell}$ has at least two entries.
  \item For any integers $i$ and $j$ in a given subset $A_{\ell}$, one has $|i-j| \ge 2$. Additionally, $1$ and $k$ lie in different sets. 
  \end{enumerate}
\end{proposition}
\begin{proof}
The result follows directly from Lemmas \ref{lem:type2} and \ref{lem:type1} and their respective proofs. Conditions (1) and (2) are necessary and sufficient for irreducibility.
\end{proof}
\begin{remark}
Proposition \ref{prop:dfg} characterizes the elements in $\mathcal{C}_{r,k}^1$.  In Lemmas \ref{lem:type2} and \ref{lem:type1}, we have seen how to deal with conditions (1) and (2) separately, which gives explicit upper bounds on $M(k,r)$. 
\end{remark}

\subsection{Some technical lemmas}
In Lemma \ref{lem:1to1}, we have seen that every path $I\in \mathcal{C}_{r,k}$ corresponds to a unique $r$-partition of $\iv{k}$ and vice versa. For simplicity, we will write $\partition (I)$ for this partition. Similarly, we define $\path (\mathcal{P})$ as the path that corresponds to the partition $\mathcal{P}$.

Next, we need the notion of {\em refined partitions}. Assume $\mathcal{P}=\{A_1,\ldots,A_r\}$ is an $r$-partition of $\iv{k}$. A partition $\{B_1,\ldots,B_{r+s}\}$ of $\iv{k}$ is called an $s$-refinement of $\mathcal{P}$ if each set in $\mathcal{P}$ is the union of some $B_i$'s. Clearly, every $s$-refinement of an $r$-partition is an $(r+s)$-partition. 

Recall the definition of a $\Delta(I,T)$ graph, its skeleton $\Delta^0(I,T)$ and $N_e(I,T)$ the number of edges of the skeleton.

We present some lemmas that help determine which $f(I,T)$ contribute most to $F(I)$. 
\begin{lemma}\label{lem:refinement}
Fix $I\in \mathcal{C}_{r,k}$ and $1\le s\le k-1$. Assume $T_1\in \mathcal{C}_{s+1,k}$ is such that $f(I,T_1)>0$ and $\Delta^0(I,T_1)$ is a tree. Then there exists a $T_2\in \mathcal{C}_{s,k}$ such that $f(I,T_2)>0$ and $\Delta^0(I,T_2)$ is a tree. 
Moreover, $T_2$ can be chosen so that $\partition(T_1)$ is a $1$-refinement of $\partition(T_2)$.
\end{lemma}
\begin{proof}
Let $I\in \mathcal{C}_{r,k}$ and $1\le s\le k-1$. Assume $T_1\in \mathcal{C}_{s+1,k}$ is such that $f(I,T_1)>0$ and $\Delta^0(I,T_1)$ is a tree. We shall construct a $T_2$  with the desired properties.

The tree $\Delta^0(I,T_1)$ has the $N_e(I,T_1)=r+s$ edges $(i,t), t\in \iv{s+1}, i \in Q_t$ for appropriate sets $Q_t\subset \iv{r}$ satisfying $ \#{Q_1}+\cdots +\#Q_{s+1}=r+s$.
Since $\Delta^0(I,T_1)$ is connected, we can find for any $T_1$-vertex $t\in \iv{s+1}$ a vertex $u\neq t$ such that $Q_t  \cap Q_u \neq \emptyset$. 

Moreover, for any $t\neq u$ the intersection of $Q_t \cap Q_u$ contains at most $1$ element. We prove this fact by contradiction. Assume that $Q_t \cap Q_u$ contained at least two elements $i$ and $j$. Then the graph with the four edges $(i,u),(i,t), (j,u),(j,t)$ is a cycle and a subgraph of $\Delta^0(I,T_1)$. Hence, $\Delta^0(I,T_1)$ could not be a tree. 

Choose $t\neq u \in \iv{s+1}$ such that $\# (Q_t \cap Q_u) =1$. We construct $T_2$ from $\partition (T_1)=\{A_1,\ldots, A_{s+1}\}$. Consider the $s$-partition of $\iv{k}$,
\begin{equation*}
\mathcal{P} = \{A_i:i\in \iv{s+1}\backslash \{t,u\}, A_t  \cup A_u\}\,.
\end{equation*}
By construction, $\partition (T_1)$ is a $1$-refinement of $\mathcal{P}$. Now set $T_2= \path(\mathcal{P}) \in \mathcal{C}_{s,k}$. $\Delta^0(I,T_2)$ is a connected graph with 
\begin{equation*}
\#{Q_1}+\cdots +\#Q_{s+1} -1 =r+s-1 
\end{equation*}
edges and thus a tree. Since the edge degrees of $\Delta(I,T_2)$ are either the same or a sum of edge degrees of $\Delta(I,T_1)$, we conclude that $f(I,T_2)>0$. 
\end{proof}
The path $T_2$ in the above construction is not necessarily unique.

The following result was proven in \cite{heiny:mikosch:2017:corr} with considerable technical effort. We provide a simple proof using graph theory.
\begin{lemma}\label{lem:9.5}
  Let $I\in \mathcal{C}_{r,k}$.  
  For any $T \in \iv{n}^k$ such that $f(I,T)>0$ we have $\# \{T\}\le k-r+1$. 
\end{lemma}
\begin{proof}
  Let $I$ be a canonical $r$-path of length $k$. $f(I,T)>0$ implies
  that each edge of the $\Delta(I,T)$ needs to appear at least twice
  which in turn implies that $N_e(I,T)$,  the number of edges of the
  skeleton $\Delta^0(I,T)$,
  is at most $k$. Because  $\Delta^0(I,T)$ is connected with   $r+ \#  \{T\}$ 
  vertices, we have
  \[  r+ \#  \{T\} \le N_e(I,T)+1 \le k+1.
  \]
\end{proof}
\begin{remark}
We note that by Lemma 3.4 in \cite{bai:silverstein:2010} there exist such $\Delta(I,T)$ graphs with $f(I,T)>0$ and $N_e(I,T)=k$. In fact, for every $I\in \mathcal{C}_{r,k}^0$, there exists a unique $T\in \mathcal{C}_{k-r+1,k}$ with this property; see the construction in \cite{bai:silverstein:2010} for details. Hence, the inequality $\# \{T\}\le k-r+1$ is sharp. Moreover, Lemma \ref{lem:refinement} then implies that for $1\le s\le k-r$ we can find at least one $T\in \mathcal{C}_{s,k}$ such that $f(I,T)>0$ and $\Delta^0(I,T)$ is a tree.
\end{remark}

Define a function $g$ by $g(\emptyset)=1$ and
\begin{equation}\label{eq:defmaxg}
g(I)= \max_{T \in \iv{n}^k}\{ \# \{T\}\,:\, f(I,T)>0  \}\,,\qquad I \in \iv{p}^k\,.
\end{equation}
Lemma~\ref{lem:9.5} can be formulated in terms of the function $g$ as follows.
\begin{lemma}\label{lem:g}
For any $I \in \mathcal{C}_{r,k}$ it holds $g(I)\le k-r+1$ with equality if and only if $I \in \mathcal{C}_{r,k}^0$.
\end{lemma}

\subsection{Finding $F(I)$}

Throughout this subsection, let $I\in \mathcal{C}_{r,k}$ with $1 \le r \le k$ and assume the conditions of Theorem \ref{thm:mainsimplified}.

Since $f(I,T_1)=f(I,T_2)$ if $T_1$ and $T_2$ are isomorphic we may
sort, analogously to \eqref{eq:okljjj},  also according to the number of distinct elements in $T$. An application of Lemma \ref{lem:g} then shows as $\nto$,
\begin{equation}\label{eq:abc}
\begin{split}
p^{r-1} F(I) &=  \sum_{s=1}^k \sum_{T\in \mathcal{C}_{s,k}} p^{r-1} n (n-1) \cdots (n-s+1) f(I,T)\\
&\sim \sum_{s=1}^{g(I)} \sum_{T\in \mathcal{C}_{s,k}} p^{r-1} n^s f(I,T)\,.
\end{split}
\end{equation}

It turns out that the quantity $N_e(I,T)$ is crucial for the order of
$f(I,T)$.  Recall that by \eqref{eq:orderfit}, $f(I,T)>0$ will be of highest order
if $N_e(I,T)$ is minimal.
We have seen in the proof of Lemma~\ref{lem:9.5} that for $T\in
\mathcal{C}_{s,k}$, $N_e(I,T)$ attains its minimum if and only if
$\Delta^0(I,T)$ is a tree.  Because  $N_e(I,T)\ge r+s-1$, we obtain
\begin{eqnarray}
p^{r-1} n^s f(I,T)&=&O(n^{r+s-1-N_e(I,T)}) \notag \\ &=&
\left\{\begin{array}{ll}
O(1) \,, & \mbox{if } \Delta^0(I,T) \text{ is a tree and } f(I,T)>0 \,,\\
O(n^{-1}) \,, & \mbox{otherwise }.
\end{array}\right. \label{eq:sega}
\end{eqnarray}
For $I \in \mathcal{C}_{k}$ set 
\begin{equation}\label{eq:CI}
\mathcal{C}_{s,k}(I) =\{T\in \mathcal{C}_{s,k} : \Delta^0(I,T) \text{ is a tree and all edges of } \Delta(I,T) \text{ possess even degrees} \}\,.
\end{equation}
Note that $\mathcal{C}_{1,k}(I)=\{ (1,\ldots,1)\}$, while for $s\ge 2$ the set $\mathcal{C}_{s,k}(I)$ might be empty. 
Thanks to \eqref{eq:sega}, \eqref{eq:abc} simplifies to 
\begin{equation*}
p^{r-1} F(I) \sim \sum_{s=1}^{g(I)} \sum_{T\in \mathcal{C}_{s,k}(I)} p^{r-1} n^s f(I,T)\,.
\end{equation*}
By virtue of Lemma~\ref{lem:refinement} we know that 
\begin{equation}\label{prop00}
\mathcal{C}_{s,k}(I)=\emptyset \quad  \text{ implies } \quad \mathcal{C}_{s+1,k}(I)=\emptyset\,, \quad s\ge 2\,.
\end{equation}
 This means that the upper summation bound $g(I)$ can be further reduced. Property \eqref{prop00} is particularly useful in computations because many sets $\mathcal{C}_{s,k}(I)$ do not have to be constructed from their definition \eqref{eq:CI} to know that they are empty. Also one can start by building the sets $\mathcal{C}_{2,k}(I), \mathcal{C}_{3,k}(I),\ldots$, i.e., the ones with the fewest number of vertices, first. 

\begin{remark}
In fact, by exhaustive enumeration we know that
$\mathcal{C}_{2,k}(I)=\emptyset$ for all irreducible $I\in
\mathcal{C}_{r,k}^1$ with length  $|I|\le 8$. This combined with the path-shortening algorithm leads to such tremendous simplifications (compared with a brute force computation), that $\E[m_k]$ can be calculated by hand in reasonable time for small $k$. If the reader wants to try, we recommend to focus on the cases $k=4,5$ when there exists only one irreducible path.
\end{remark}

Assume we have already constructed $\mathcal{C}_{s,k}(I)$ and that it is nonempty. As long as $s<g(I)$, it is possible that the next set $\mathcal{C}_{s+1,k}(I)$ is nonempty. Fortunately, the proof of Lemma~\ref{lem:refinement} provides an explicit  construction of  potential paths in $\mathcal{C}_{s+1,k}(I)$ as paths corresponding to $1$-refinements of partitions of paths in $\mathcal{C}_{s,k}(I)$. In other words, any $T_1\in \mathcal{C}_{s+1,k}(I)$ is the path generated by some $1$-refinement of $\partition(T_2)$ for some $T_2\in \mathcal{C}_{s,k}(I)$.

As regards to the task of determining the sets $\mathcal{C}_{s,k}(I)$, the worst possible scenario happens when $g(I)=k-r+1$, or equivalently $I\in \mathcal{C}_{r,k}^0$. In this situation, $\mathcal{C}_{k-r+1,k}(I)\neq \emptyset$ and hence all other sets too are nonempty. Fortunately, in this situation Lemma~\ref{prop:PSI} gives $F(I)=n^{1-r}$ so that \eqref{eq:abc} is superfluous. 

We summarize the preliminary results of this subsection in the following statement. 
For any $I\in \mathcal{C}_{r,k}$  one has as $\nto$,
\begin{equation}\label{eq:aaa}
p^{r-1} F(I) \sim \sum_{s=1}^{t^{\star}(I)} \sum_{T\in \mathcal{C}_{s,k}(I)} p^{r-1} n^s f(I,T)
\end{equation}
with $t^{\star}(I)=\min \{1\le s \le g(I) : \mathcal{C}_{s+1,k}(I)= \emptyset \}$. All the terms in the sum on the \rhs~ of \eqref{eq:aaa} are of order $O(1)$.

It remains to provide an explicit formula for the limit of $p^{r-1} n^s f(I,T)$. From \eqref{eq:fI2} we get 
\begin{equation}\label{eq:eraw}
  f(I,T) 
   =  \prod_{i=1}^r \E\Big[ \prod_{t\in \mathcal{T}_i(I,T)}  Y_{it}^{m_{it}(I,T)} \Big] \,,
 \end{equation}
where $m_{it}(I,T)$ is the degree of edge $(i,t)\in \Delta(I,T)$ and
$\mathcal{T}_i(I,T)$ denotes the set of neighbours of an $I$-vertex $i$, i.e., $\mathcal{T}_i(I,T)=\{t\in \{T\}: (i,t) \in \Delta(I,T)\}$.
Let $d_i(I,T)= \# \mathcal{T}_i(I,T) $ be the degree of $i$ (in $\Delta(I,T)$).
By \eqref{moment} we have as $\nto$,
\begin{equation}\label{moment1}
  \E\Big[ \prod_{t\in \mathcal{T}_i(I,T)}  Y_{it}^{m_{it}(I,T)} \Big]
  \sim 
  \frac{ \Gamma(d_i)}{n^{d_i}} \frac{
    \Big(\frac{\alpha}{2}\Big)^{d_i-1} \prod_{t\in \mathcal{T}_i} \Gamma(\tfrac{m_{it}-\alpha}{2})}{  \Big(\Gamma(1-\alpha/2)\Big)^{d_i} \Gamma(N_i)}\,,
\end{equation}
where $N_i=N_i(I)$ counts the number of appearances of the integer $i$
in the path $I$. Here and below the dependence on $(I,T)$ is sometimes
removed in the notations for the sake of clarity when no ambiguity
is possible.
Since $\Delta^0(I,T)$ is a tree, it follows that 
\begin{equation}\label{eq:nerr}
\sum_{i=1}^r d_i(I,T) = N_e(I,T) =r+s-1\,.
\end{equation}
Thanks to \eqref{eq:eraw}, \eqref{moment1} and \eqref{eq:nerr}, one sees that
\begin{equation}\label{eq:edgfe}
\begin{split}
  p^{r-1} n^s f(I,T) &\sim \Big(\frac{p}{n} \Big)^{r-1} \Big(\frac{\alpha}{2} \Big)^{s-1} \Big(\Gamma(1-\alpha/2)\Big)^{-(r+s-1)}\\
  & \qquad \prod_{i=1}^r \frac{\Gamma(d_i)}{\Gamma(N_i)}
  \prod_{(i,t)\in \Delta(I,T)} \Gamma\Big(\frac{m_{it}-\alpha}{2}\Big)\,.
\end{split}
\end{equation}

A combination of \eqref{eq:aaa} and \eqref{eq:edgfe} proves the following result.
\begin{proposition}\label{prop:calcF}
  Assume the conditions of Theorem \ref{thm:mainsimplified}. For any $I\in \mathcal{C}_{r,k}$  one has 
\begin{equation}\label{FI}
\begin{split}
  \lim_{\nto} p^{r-1} F(I) &= \Big(\frac{\gamma}{\Gamma(1-\alpha/2)} \Big)^{r-1} \frac{2}{\alpha}
  \sum_{s=1}^{t^{\star}(I)} \Big( \frac{\alpha/2}{\Gamma(1-\alpha/2)} \Big)^s
  \\
  & \qquad \sum_{T\in \mathcal{C}_{s,k}(I)} \left(\prod_{i=1}^r \frac{\Gamma(d_i)}{\Gamma(N_i)} \right)
\prod_{(i,t)\in \Delta(I,T)} \Gamma\Big(\frac{m_{it}-\alpha}{2}\Big)\,,
\end{split}
\end{equation}
where
\begin{eqnarray*}
  \mathcal{C}_{s,k}(I) &=&\{T\in \mathcal{C}_{s,k} : \Delta^0(I,T) \text{ has } \# \{I\}+s-1 \text{ edges } (i,t) \text{ with even degrees } m_{it}(I,T) \},\\
  t^{\star}(I) &=&\min \{1\le s \le g(I) : \mathcal{C}_{s+1,k}(I)= \emptyset \}.
\end{eqnarray*}
\end{proposition}

\subsection{Some examples}\label{sec:2.5}

To better understand Proposition \ref{prop:calcF} and its notation, we provide some examples.

\begin{example}\label{ex:1212}
Consider the path $I=(1,2,1,2)$. We discuss various ways of calculating $p\, F(I)$. 
First, a direct calculation using the symmetry of $Y_{it}$ shows that
\begin{equation*}
\begin{split}
F(1,2,1,2)&=
\sum_{t_1,\ldots,t_4=1}^n \E[ Y_{1t_1}Y_{1t_2}Y_{1t_3}Y_{1t_4} 
Y_{2t_1}Y_{2t_2}Y_{2t_3}Y_{2t_4}]\\
&= \sum_{t_1,\ldots,t_4=1}^n (\E[ Y_{1t_1}Y_{1t_2}Y_{1t_3}Y_{1t_4} ])^2\\
&= \sum_{t_1=1}^n (\E[ Y_{1t_1}^4])^2 + 3 \sum_{t_1\neq t_2=1}^n (\E[ Y_{1t_1}^2Y_{1t_2}^2 ])^2\\
&\sim  \frac{1}{n} (n\E[ Y_{11}^4])^2 \sim  \frac{1}{n} (1-\alpha/2)^2\,.
\end{split}
\end{equation*}
Hence, $\lim_{\nto}  p\, F(I)= \gamma (1-\alpha/2)^2$.

Next, we are going to apply Proposition \ref{prop:calcF}. By construction, $\mathcal{C}_{1,4}(I)=\{(1,1,1,1)\}$. It is easily checked that $\mathcal{C}_{2,k}(I)=\emptyset$ which implies $t^{\star}(I)=1$. The edges of $\Delta(I,(1,1,1,1))$ have degree $4$. Therefore we have 
\begin{equation*}
\begin{split}
\lim_{\nto} p^{r-1} F(I) &= \frac{\gamma}{\Gamma(1-\alpha/2)}  \frac{2}{\alpha}
  \frac{\alpha/2}{\Gamma(1-\alpha/2)} 
  \Big(\frac{\Gamma(1)}{\Gamma(2)}\Big)^2
\Big( \Gamma\big(\tfrac{4-\alpha}{2}\big)\Big)^2=\gamma (1-\alpha/2)^2\,.
\end{split}
\end{equation*}
\end{example}

For longer paths a combination of path-shortening and Proposition \ref{prop:calcF} is useful.
\begin{example}
Consider the path $I=(1,1,2,1,3,3,4,5,4,6,7,7,3,8,9,8,6,1)\in \mathcal{C}_{9,18}$. A direct calculation as in Example \ref{ex:1212} would be quite tedious. Using the path-shortening Lemma \ref{prop:PSI} we obtain
 \begin{equation*}
\begin{split}
p^8 F(I)&=p^8 F(S(I))\, n^{-\simples(I)} = p^8 F(3,6,3,6)\, n^{-7} \\
&= \Big(\frac{p}{n}\Big)^7 p F(1,2,1,2) \to \gamma^8 (1-\alpha/2)^2\,.
\end{split}
\end{equation*}
\end{example}

Finally, we want to provide nontrivial examples of the sets $\mathcal{C}_{s,k}(I)$. By nontrivial we mean $\mathcal{C}_{2,k}(I)\neq \emptyset$ for which $k$ is required to be at least $9$.  The defining properties are checked by counting the number of edges of the $\Delta(I,T)$ graphs and their degrees. 

 \begin{example}\label{ex:4.9}
\begin{figure}[h!]
  \centering
  \begin{minipage}{0.45\linewidth}
    \vskip1.95mm
    \includegraphics[width=\linewidth]{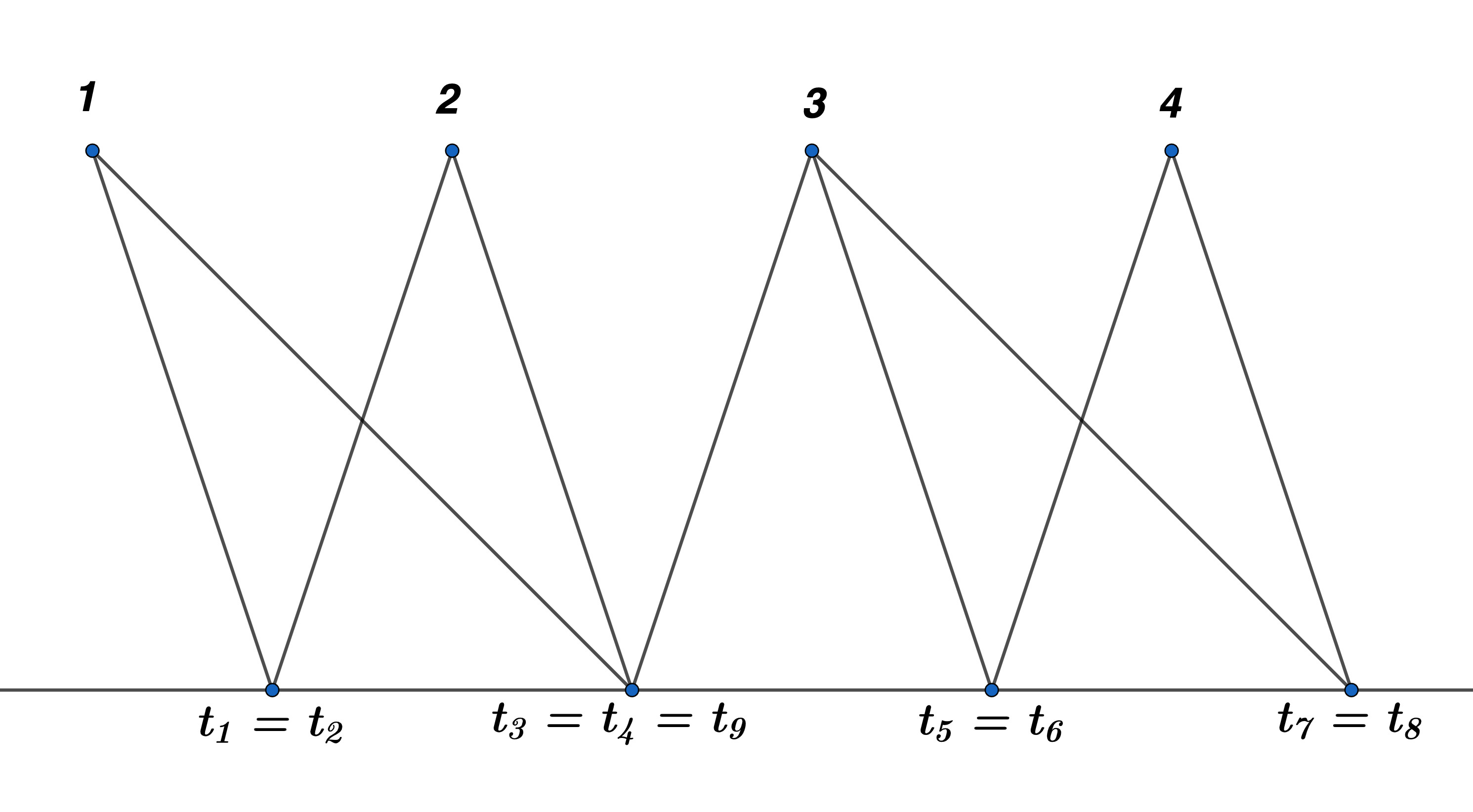} 
  \end{minipage}
  \quad
  \begin{minipage}{0.45\linewidth}
    \includegraphics[width=\linewidth]{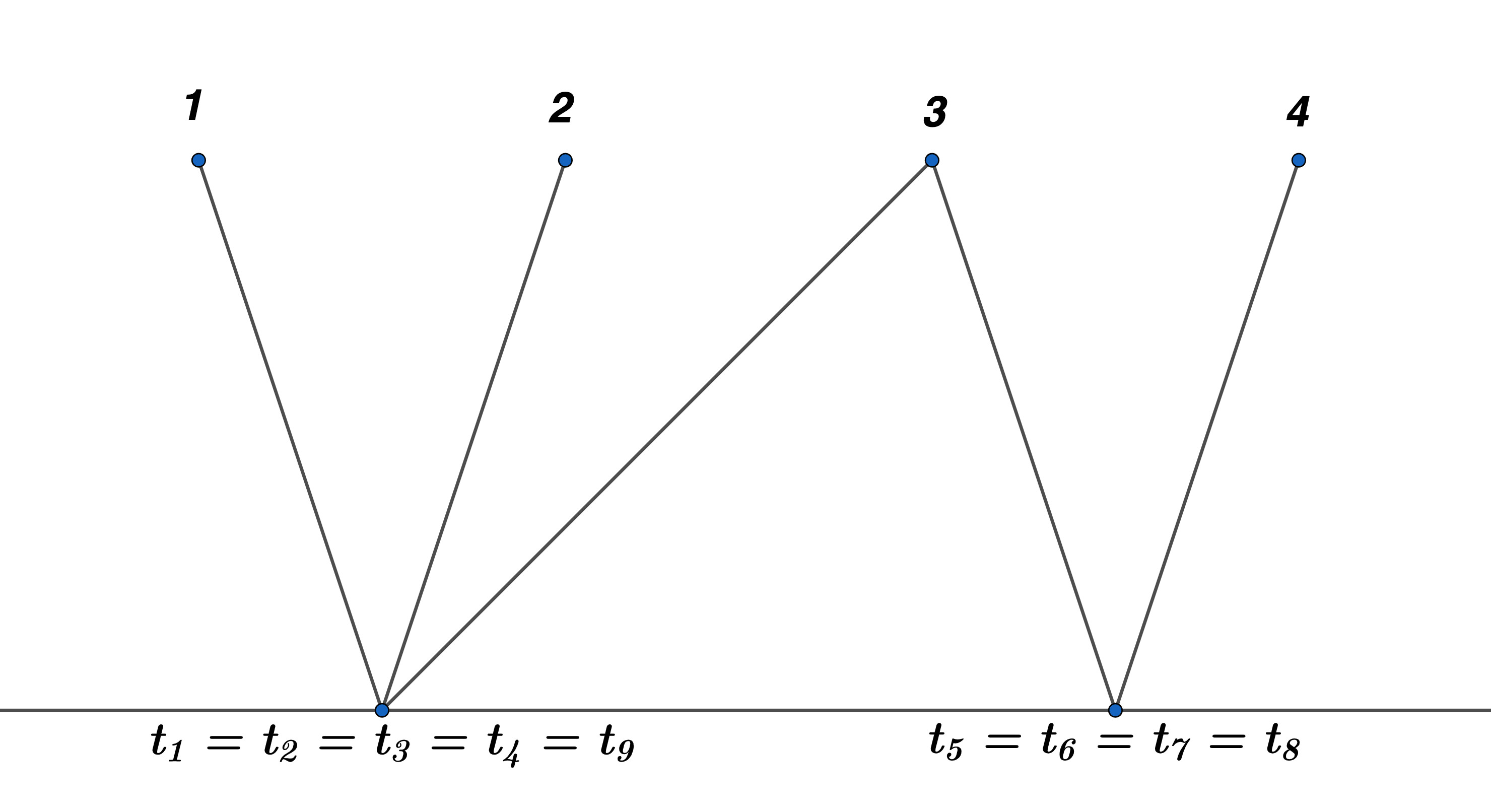} 
  \end{minipage}
  \caption{Graphs $\Delta^0(I,T_2)$ (left) and $\Delta^0(I,T_3)$
    (right) for $I=(1,2,1,2,3,4,3,4,3)$.}
	 \label{fig:plotex1}
\end{figure}

Our goal is to find $\mathcal{C}_{2,9}(I)$ for the irreducible path
$I=(1,2,1,2,3,4,3,4,3)$. We set $T=(t_1,t_2,\ldots, t_9)$ and list the
edges of the $\Delta(I,T)$ graph as follows. In the left column we
list the $I$-vertices, and the right column lists  respective
neighbours ($T$-vertices).
\begin{center}
  \begin{tabular}{c| l | c |c} 
    $I$-vertex $i$ & edges $(i, \cdot)$ & edge degrees even? & $\Delta^0(I,T)$ tree?\\
    \hline
    1 & $t_1, t_2,t_3,t_9$ & no & no\\
    2 & $t_1, t_2,t_3,t_4$ & no & \\
    3 & $t_5, t_6,t_7,t_8, t_4,t_9$ & no & \\
    4 & $t_5, t_6,t_7,t_8$ & no & 
  \end{tabular}
\end{center}
From the first two rows we deduce that $t_4=t_9$ is necessary to generate even edge degrees. Setting $T_1=(t_1,\ldots,t_8,t_4)$ and drawing a box $\boxed{\phantom{t}}$ around the edges with even degrees we obtain the table
\begin{center}
    \begin{tabular}{c| l | c |c} 
		$I$-vertex $i$ & edges $(i, \cdot)$ & edge degrees even? & $\Delta^0(I,T_1)$ tree?\\
		\hline
       1 & $t_1, t_2,t_3,t_4$ & no & no\\
			 2 & $t_1, t_2,t_3,t_4$ & no & \\
			 3 & $t_5, t_6,t_7,t_8, \boxed{t_4}$ & no & \\
       4 & $t_5, t_6,t_7,t_8$ & no & 
    \end{tabular}
  \end{center}
To ensure edge degrees $2$, we need to form two pairs in the quadruples $(t_1,t_2,t_3,t_4)$ and  $(t_5, t_6,t_7,t_8)$, respectively. There are $9$ possibilities. We show the characteristics of the graph for $t_1=t_2, t_3=t_4$ and $t_5=t_6, t_7=t_8$, so $T_2=(t_1,t_1,t_3,t_3,t_5,t_5,t_7, t_7,t_8,t_4)$, in the next table:
\begin{center}
    \begin{tabular}{c| l | c |c} 
		$I$-vertex $i$ & edges $(i, \cdot)$ & edge degrees even? & $\Delta^0(I,T_2)$ tree?\\
		\hline
       1 & $\boxed{t_1},\boxed{t_3}$ & yes & no\\
			 2 & $\boxed{t_1},\boxed{t_3}$ & yes & \\
			 3 & $\boxed{t_5},\boxed{t_7}, \boxed{t_3}$ & yes & \\
       4 & $\boxed{t_5},\boxed{t_7}$ & yes & 
    \end{tabular}
  \end{center}
From this table or the top panel of Figure~\ref{fig:plotex1} it is obvious that $\Delta^0(I,T_2)$ contains $2$ cycles. The only way to remove them and fulfil the tree requirement is to choose $t_1=t_3$ and $t_5=t_7$; see Figure~\ref{fig:plotex1} bottom. The other $8$ possibilities of building pairs ultimately lead to the same path structure. Hence, the canonical representative of $T_3=(t_1,t_1,t_1,t_1,t_5,t_5,t_5,t_5,t_1)$ is the only element of $\mathcal{C}_{2,9}(I)$, i.e.,
\begin{equation*}
\mathcal{C}_{2,9}(I)= \{(1,1,1,1,2,2,2,2,1)\}\,.
\end{equation*}
\end{example}

\subsection{Variance bound}\label{variancebound}

Assume the conditions of Theorem \ref{thm:mainsimplified}. We derive an upper bound for the variance of $m_k$.
For a path  $I=(i_1, i_2, \ldots, i_k)$ with vertices in $ \iv{p}$ 
and a path
$T=(t_1, i_2, \ldots, t_k)$ with vertices in $ \iv{n}$, we define 
\begin{equation}\label{eq:ftilde}
  \tilde f(I,T) = \tilde f_n(I,T) = 
  Y_{i_1t_1}  Y_{i_2t_1} Y_{i_2t_2}  Y_{i_3t_2} Y_{i_3t_3} \cdots  Y_{i_kt_k}    Y_{i_{k+1}t_k}\,.
	\end{equation}
	Then we have 
	\[  m_k =  \frac1p \sum_{I \in {\iv{p}}^k} \sum_{T \in \iv{n}^k} \tilde f(I,T)\,
\] 
and consequently the variance of $m_k$ can be written as
\begin{equation*}
\begin{split}
\Var(m_k)= \frac{1}{p^2} \sum_{I,J \in {\iv{p}}^k} \sum_{T_1,T_2 \in \iv{n}^k}
\Big(\E\big[\tilde f(I,T_1)  \tilde f(J,T_2)\big]- \E\big[\tilde f(I,T_1) \big] \E\big[ \tilde f(J,T_2)\big]\Big)\,.
\end{split}
\end{equation*}
If $\{I\}$, i.e. the set of distinct elements of $I$, and $\{J\}$ are disjoint, then $\tilde f(I,T_1)$ and $\tilde f(J,T_2)$ are independent which implies that $\E\big[\tilde f(I,T_1)  \tilde f(J,T_2)\big]- \E\big[\tilde f(I,T_1) \big] \E\big[ \tilde f(J,T_2)\big]=0$. Therefore we obtain the bound
\begin{equation}\label{eq:sdsdsdd}
\begin{split}
\Var(m_k)
&\le \frac{1}{p^2} \sum_{\substack{I,J \in {\iv{p}}^k\\ \{I\}\cap \{J\} \neq \emptyset}}  \sum_{T_1,T_2 \in \iv{n}^k}
\E\big[\tilde f(I,T_1)  \tilde f(J,T_2)\big] \\
&\le \frac{1}{p^2} \sum_{r=1}^{2k} \sum_{s=1}^{2k} p^r n^s \sum_{\substack{(I,J)\in \mathcal{C}_{r,2k} \\ \{I\}\cap \{J\} \neq \emptyset}}
 \sum_{(T_1,T_2)\in \mathcal{C}_{s,2k}} \E\big[\tilde f(I,T_1)  \tilde f(J,T_2)\big]\,,
\end{split}
\end{equation}
where we replaced each $(I,J)$ and $(T_1,T_2)$ by their canonical representatives in the last line.
 Analogously to \eqref{eq:orderfit}, the asymptotic behaviour of $\E\big[\tilde f(I,T_1)  \tilde f(J,T_2)\big]$ can be expressed in terms of the graph $\widetilde \Delta(I,J,T_1,T_2)$ which is is defined as the union of $\Delta(I,T_1)$ and $\Delta(J,T_2)$. That is, its set of vertices and edges is the union of the sets of vertices and edges, respectively, of $\Delta(I,T_1)$ and $\Delta(J,T_2)$. Since $\Delta(I,T_1)$ and $\Delta(J,T_2)$ are connected graphs we observe that $\widetilde \Delta(I,J,T_1,T_2)$ is a connected graph for all $T_1,T_2\in \iv{n}^k$ if and only if $\{I\}\cap \{J\} \neq \emptyset$. Thus, all the graphs $\widetilde \Delta(I,J,T_1,T_2)$ associated with $\E\big[\tilde f(I,T_1)  \tilde f(J,T_2)\big]$ in \eqref{eq:sdsdsdd} are connected. It suffices to consider $\widetilde \Delta(I,J,T_1,T_2)$ with even edge degrees since otherwise $\E\big[\tilde f(I,T_1)  \tilde f(J,T_2)\big]=0$.

By \eqref{moment} (compare also with \eqref{eq:orderfit}), we get 
\begin{equation}\label{eq:orderfit11}
  \E\big[\tilde f(I,T_1)  \tilde f(J,T_2)\big] \sim  \theta(I,J,T_1,T_2) \, n^{-N_e(I,J,T_1,T_2)} \,, \quad \nto\, ,
\end{equation}
for some positive constant $\theta(I,J,T_1,T_2)$. Here $N_e(I,J,T_1,T_2)$ denotes the number of edges of $\widetilde \Delta^0(I,J,T_1,T_2)$, the skeleton of $\widetilde \Delta(I,J,T_1,T_2)$.

Now, if $(I,J)\in \mathcal{C}_{r,2k}$, $\{I\}\cap \{J\} \neq \emptyset$ and $(T_1,T_2)\in \mathcal{C}_{s,2k}$, then $N_e(I,J,T_1,T_2)\ge r+s-1$ since $\widetilde \Delta^0(I,J,T_1,T_2)$ is a connected graph. In combination with \eqref{eq:sdsdsdd} and \eqref{eq:orderfit11}, this yields that 
\begin{equation*}
\Var(m_k) 
= O(n^{-1})\,, \qquad \nto\,.
\end{equation*}

\section{Completion of the proof of Theorem~\protect\ref{thm:mainsimplified}}\label{sec:mainresult}

Recall a few  important notations that were introduced in Sections \ref{sec:1}-\ref{sec:F(I)}:
\begin{eqnarray*}
  \beta_k(\gamma) & = & \text{$k$-th \MP  moment (see \eqref{eq:momentsmp})}\, ;\\
	\mathcal{C}_{r,k}&=& \{ \text{canonical $r$-paths of length $k$} \}\,;\\
  \mathcal{C}_{s,k}(I) &=&\{T\in \mathcal{C}_{s,k} : \Delta^0(I,T) \text{ has } \#\{I\}+s-1 \text{ edges } (i,t) \text{ with even degrees } m_{it}(I,T) \};\\
  t^{\star}(I) &=&\min \{1\le s \le k-\#\{I\} : \mathcal{C}_{s+1,k}(I)= \emptyset \};\\
  N_i(I) &=& \text{number of appearances of the vertex } i \text{ in the path } I\,;\\
  d_i(I,T) &=& \text{number of neighbours ($T$-vertices) of an
               $I$-vertex $i$ in  $\Delta(I,T)$}.
\end{eqnarray*}
To shorten notation, we will write $\tI$ for the canonical representative of $S(I)$. Finally, we define the sets 
\begin{equation*}
\mathcal{C}_{r,k}^{(q)} = \{I\in \mathcal{C}_{r,k} : 0\le \simples(I)=q \le r-2\}\,,\qquad 0\le q \le r-2\,.
\end{equation*}

%

Now we complete the proof of Theorem~\protect\ref{thm:mainsimplified}
with the following formula for the $k$-th moments of the limiting
\ahmplaw\ $H_{\alpha,\gamma}$:
\begin{equation}\label{mainresult}
    \begin{split}
\mu_k(\alpha,\gamma) &= \beta_k(\gamma) +  \frac{2}{\alpha} \sum_{r=2}^{k-2} \gamma^{r-1} \sum_{q=0}^{r-2} (\Gamma(1-\alpha/2))^{-r+q+1} \\
& \quad \sum_{I\in \mathcal{C}_{r,k}^{(q)}} 
\sum_{s=1}^{t^{\star}(\tI)} \Big( \frac{\alpha/2}{\Gamma(1-\alpha/2)} \Big)^s
  \sum_{T\in \mathcal{C}_{s,|\tI|}(\tI)} \left(\prod_{i=1}^{r-q} \frac{\Gamma(d_i(\tI,T))}{\Gamma(N_i(\tI))} \right)\\
& \quad \prod_{(i,t)\in \Delta(\tI,T)} \Gamma\Big(\frac{m_{it}(\tI,T)-\alpha}{2}\Big)\,.
    \end{split}
\end{equation}
In the course of deduction we will also see that 
every path $\tI$ in \eqref{mainresult} lies in the set 
\begin{equation}\label{eq:setI}
  \bigcup_{s=2}^{\lfloor k/2 \rfloor} \bigcup_{\ell=4}^k \mathcal{C}_{s,\ell}^1\,.
\end{equation}

To proceed, note that
  weak convergence in probability follows from
\begin{itemize}
\item[$(i)$] For all $k\ge 1$, $\lim_{\nto}
  \E[m_k]=\mu_k(\alpha,\gamma)$, ~  and
\item[$(ii)$] $\lim_{\nto}\Var(m_k)=0$\,,
\end{itemize}
where $m_k = \int x^k d F_\bfR (x)$. 
In Section \ref{variancebound}, we proved that $\Var(m_k)=O(n^{-1})$ as $\nto$ which implies $(ii)$. 

Starting in \eqref{eq:oklj}, we have shown over the course of Sections \ref{sec:moments} and \ref{sec:F(I)} that
\begin{equation}\label{eq:degf}
  \lim_{\nto} \E[m_k] =  \beta_k(\gamma) + \lim_{\nto} \sum_{r=1}^k   \sum_{I\in \mathcal{C}_{r,k}: S(I)\neq \emptyset} p^{r-1} n^{-\simples(I)} \, F(S(I)).
\end{equation}
Observe that the condition $0\le \simples(I) \le r-2$ is equivalent to $S(I)\neq \emptyset$. Using the notation $\wt I= S(I)$ and the definition of $\mathcal{C}_{r,k}^{(q)}$, we have 
\begin{equation}\label{eq:edr}
\begin{split}
\sum_{r=1}^k    \sum_{I\in \mathcal{C}_{r,k}: S(I)\neq \emptyset} p^{r-1} n^{-\simples(I)} \, F(S(I))
&= \sum_{r=2}^{k-2}  \sum_{q=0}^{r-2}  \Big(\frac{p}{n}\Big)^q \sum_{I\in \mathcal{C}_{r,k}^{(q)}} p^{r-q-1}  \, F(\tI)\,.
\end{split}
\end{equation}
The limit of $p^{r-q-1}  \, F(\tI)$ is then calculated via Proposition
\ref{prop:calcF}. This implies claim  $(i)$. 

Next, for $I\in \mathcal{C}_{r,k}$ we have $\tI \in \mathcal{C}_{r-\simples(I), |S(I)|}^1$ and therefore every path $\tI$ in \eqref{eq:edr} lies in the set 
$\bigcup_{s=2}^{\lfloor k/2 \rfloor} \bigcup_{\ell=4}^k \mathcal{C}_{s,\ell}^1$,
which is relatively small; see Section \ref{sec:counting} for details.

Finally, to ensure that the sequence of moments
$(\mu_k(\alpha,\gamma))_{k\ge 1}$ in \eqref{mainresult} uniquely determines a probability
distribution,
we check  the Carleman condition, that is
\begin{equation}\label{eq:carleman}
  \sum_{k\ge 1} (\mu_{2k}(\alpha,\gamma))^{-\frac1{2k}}=\infty\,.
\end{equation}
From \eqref{eq:degf} we have
\[ \beta_k(\gamma) \le \mu_k(\alpha,\gamma)\le \sum_{r=1}^k \gamma^{r-1} \# \mathcal{C}_{r,k}.
\]
By Lemma~\ref{lem:crk} and for $1\le r\le k$, 
\[ \# \mathcal{C}_{r,k} =  B(k,r) \le \frac12 \binom{k}{r}r^{k-r}
\le \frac12 \binom{k}{r}k^{k-r} ,
\]
where the first upper bound for $B(k,r)$ is well-known (see \cite{rennie:dobson:1969}). Therefore,
\[  \mu_k(\alpha,\gamma)\le \sum_{r=1}^k \gamma^{r-1}  \frac12
\binom{k}{r}k^{k-r}\le (2\gamma)^{-1} (\gamma+k)^k,
\]
and
\[  (\mu_{2k}(\alpha,\gamma))^{-\frac1{2k}} \ge
(2\gamma)^{\frac1{2k}}(\gamma+2k)^{-1}\sim \frac{1}{\gamma+2k} \,, \qquad k\to \infty\,.
\]
The Carleman  condition \eqref{eq:carleman} is satisfied.
The proof of Theorem~\ref{thm:mainsimplified} is complete.

\subsection{Computation of the limiting moments $\mu_k(\alpha,\gamma)$}\label{sec:k45}

Formula \eqref{mainresult} is explicit and requires some counting that can be implemented in mathematical software. For small values of $k$, it is feasible to evaluate \eqref{mainresult} without computing support. We find the first 5 moments $\mu_k(\alpha,\gamma)$. If $k=1,2,3$, we immediately get $\mu_k(\alpha,\gamma)=\beta_k(\gamma)$. Let us turn to $k\in \{4,5\}$. By \eqref{eq:setI}, we have 
\begin{equation*}
\tI \in\bigcup_{s=2}^{\lfloor k/2 \rfloor} \bigcup_{\ell=4}^k \mathcal{C}_{s,\ell}^1 = \mathcal{C}_{2,4}^1= \{(1,2,1,2)\}\,.
\end{equation*}
From Example \ref{ex:1212} we know that $t^{\star}(1,2,1,2)=1$. Since
$\mathcal{C}_{1,|\tI|}(\tI)=\{(1,1,1,1)\}$ it suffices to study the
graph $\Delta(\tI,T)$ with $T=(1,\ldots,1)$. From Figure~\ref{fig:ex5} one can see that,
\begin{equation*}
  \sum_{T\in \mathcal{C}_{1,|\tI|}(\tI)} \Big(\prod_{i=1}^{r-q} \underbrace{\frac{\Gamma(d_i(\tI,T))}{\Gamma(N_i(\tI))}}_{\Gamma(1)/\Gamma(2)} \Big) \prod_{(i,t)\in \Delta(\tI,T)} \underbrace{\Gamma\Big(\frac{m_{it}(\tI,T)-\alpha}{2}\Big)}_{\Gamma(1-\alpha/2)} = (1-\alpha/2)^2 (\Gamma(1-\alpha/2))^2\,.
\end{equation*}
\begin{figure}[htp]
  \centering
  \includegraphics[width=0.6\linewidth]{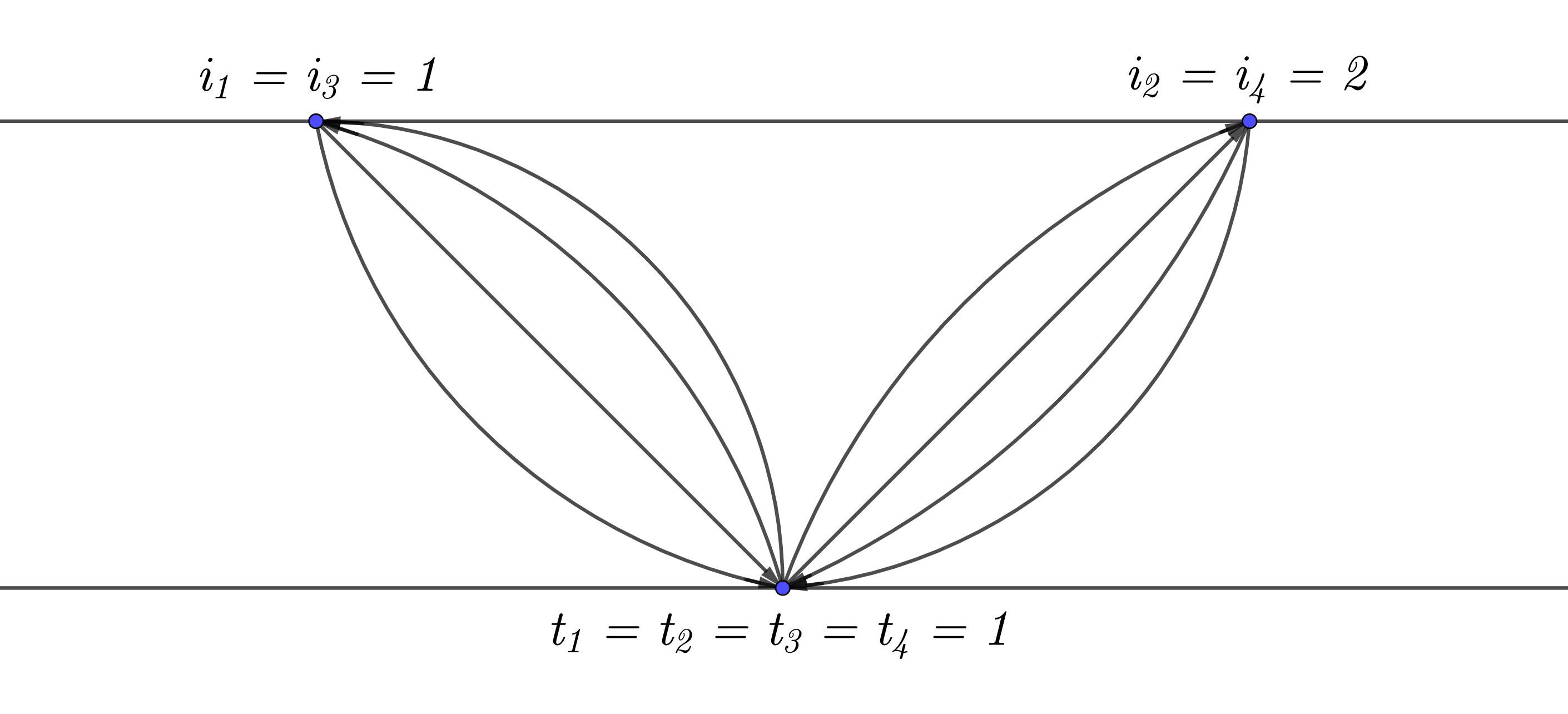}
  \caption{Graph $\Delta(\tI,T)$ with $\tI=(1,2,1,2)$ and
    $T=(1,1,1,1)$.}
  \label{fig:ex5}
\end{figure}

Hence, \eqref{mainresult} reads as
\begin{equation*}
\begin{split}
\mu_k(\alpha,\gamma) &= \beta_k(\gamma) +  \frac{2}{\alpha} \sum_{r=2}^{k-2} \gamma^{r-1} \sum_{q=0}^{r-2} (\Gamma(1-\alpha/2))^{-r+q+1} \\
& \quad \sum_{I\in \mathcal{C}_{r,k}^{(q)}} 
\frac{\alpha/2}{\Gamma(1-\alpha/2)} (1-\alpha/2)^2 (\Gamma(1-\alpha/2))^2\\
&= \beta_k(\gamma) + (1-\alpha/2)^2 \sum_{r=2}^{k-2} \gamma^{r-1} \sum_{q=0}^{r-2} (\Gamma(1-\alpha/2))^{-r+q+2} \,\#\mathcal{C}_{r,k}^{(q)}\,.
\end{split}
\end{equation*}
In view of $\mathcal{C}_{2,4}^{(0)}=\{(1,2,1,2)\}$, this gives 
\begin{equation*}
\mu_4(\alpha,\gamma)= \beta_4(\gamma) + (1-\alpha/2)^2 \gamma\,.
\end{equation*}
In order to find $\mu_5$, we need to construct the sets
\begin{equation*}
\mathcal{C}_{2,5}^{(0)}= \begin{Bmatrix}
(1,1,2,1,2),\\
(1,2,1,1,2),\\
(1,2,1,2,1),\\
(1,2,2,1,2),\\
(1,2,1,2,2)
\end{Bmatrix}\,, \quad
\mathcal{C}_{3,5}^{(0)}=\emptyset\,,
\quad \text{ and } \quad 
\mathcal{C}_{3,5}^{(1)} = \begin{Bmatrix}
(1,2,3,1,2),\\
(1,2,1,3,2),\\
(1,2,1,2,3),\\
(1,2,3,1,3),\\
(1,2,3,2,3)
\end{Bmatrix}\,.
\end{equation*}
This shows that 
\begin{equation*}
\mu_5(\alpha,\gamma)= \beta_5(\gamma) + (1-\alpha/2)^2 (5 \gamma + 5 \gamma^2)\,,
\end{equation*}
where we used that $\tI=(1,2,1,2)$ for $I\in \mathcal{C}_{2,5}^{(0)} \cup \mathcal{C}_{3,5}^{(1)}$.

\section{Proof of Theorem~\protect\ref{thm:limit} for the boundary cases}\label{sec:boundary} 

The proof of Theorem~\protect\ref{thm:limit} is decomposed into two
lemmas.

\begin{lemma}\label{lem:alpha02}
Let  $\mu_k(\alpha,\gamma)$ be defined in \eqref{mainresult}. For $k\ge 1$ it holds 
$$\lim_{\alpha \to 0^+} \mu_k(\alpha,\gamma) =\frac{1}{\gamma} \sum_{r=1}^{k} \gamma^{r} B(k,r) \quad \text{ and } \quad \lim_{\alpha \to 2^-} \mu_k(\alpha,\gamma) =\beta_k(\gamma)\,,$$
where $B(k,r)$ is the Stirling number of the second kind defined in \eqref{eq:stirling2} and $\beta_k(\gamma)$ is the $k$-th \MP moment.
\end{lemma}
\begin{proof}
For $\alpha\in (0,2)$ and $k\ge 1$, we use the decomposition $\mu_k(\alpha,\gamma)=\beta_k(\gamma) + d_k(\alpha,\gamma)$.
 For $k\in \{1,2,3\}$ we have $ \mu_k(\alpha,\gamma) =\beta_k(\gamma)=\frac{1}{\gamma} \sum_{r=1}^{k} \gamma^{r} B(k,r)$ by \eqref{eq:stirling2}. 

Hence, it suffices to assume $k\ge 4$. We have
\begin{equation*}
    \begin{split}
\lim_{\alpha \to 0^+} d_k(\alpha,\gamma)&=  \lim_{\alpha \to 0^+} \sum_{r=2}^{k-2} \gamma^{r-1} \sum_{q=0}^{r-2}  
 \sum_{I\in \mathcal{C}_{r,k}^{(q)}} 
\sum_{s=1}^{t^{\star}(\tI)} (\alpha/2)^{s-1}
  \sum_{T\in \mathcal{C}_{s,|\tI|}(\tI)} \left(\prod_{i=1}^{r-q} \frac{\Gamma(d_i(\tI,T))}{\Gamma(N_i(\tI))} \right)\\
& \quad \prod_{(i,t)\in \Delta(\tI,T)} \Gamma\Big(\frac{m_{it}(\tI,T)}{2}\Big)\\
&= \sum_{r=2}^{k-2} \gamma^{r-1} \sum_{q=0}^{r-2}  
 \sum_{I\in \mathcal{C}_{r,k}^{(q)}} 
  \sum_{T\in \mathcal{C}_{1,|\tI|}(\tI)} \left(\prod_{i=1}^{r-q} \frac{\Gamma(d_i(\tI,T))}{\Gamma(N_i(\tI))} \right) \prod_{(i,t)\in \Delta(\tI,T)} \Gamma\Big(\frac{m_{it}(\tI,T)}{2}\Big)\\
	&= \sum_{r=2}^{k-2} \gamma^{r-1} \sum_{q=0}^{r-2}  
 \sum_{I\in \mathcal{C}_{r,k}^{(q)}} 
   \left(\prod_{i=1}^{r-q} \frac{\Gamma(1)}{\Gamma(N_i(\tI))} \right)  \left(\prod_{i=1}^{r-q} \Gamma(N_i(\tI))\right)\\
	&= \sum_{r=2}^{k-2} \gamma^{r-1} (\#\mathcal{C}_{r,k}- \#\mathcal{C}_{r,k}^0)\\
	&= \sum_{r=2}^{k-2} \gamma^{r-1} \left[B(k,r)-  \frac{1}{r} \binom{k}{r-1} \binom{k-1}{r-1}\right]\,,
\end{split}
\end{equation*}
where Lemma \ref{lem:crk} and \eqref{lem:lemma3.4} were used for the last two equalities, respectively.
Since $\beta_k(\gamma)=\sum_{r=1}^{k} \gamma^{r-1} \#\mathcal{C}_{r,k}^0$ this implies that 
\begin{equation*}
    \begin{split}
\lim_{\alpha \to 0^+} \mu_k(\alpha,\gamma)&=   1+\binom{k}{2} \gamma^{k-2} + \gamma^{k-1} +\sum_{r=2}^{k-2} \gamma^{r-1} B(k,r)\\
&=\frac{1}{\gamma} \sum_{r=1}^{k} \gamma^{r} B(k,r)
\end{split}
\end{equation*}
since $B(k,k)=1, B(k,k-1)=\binom{k}{2}$ and $B(k,1)=1$.

Next, we turn to the limit $\alpha \to 2^-$ and observe that $\lim_{\alpha \to 2^-}\Gamma(1-\alpha/2)=\infty$. We have 
\begin{equation*}
    \begin{split}
\lim_{\alpha \to 2^-} d_k(\alpha,\gamma)&=  \lim_{\alpha \to 2^-} \sum_{r=2}^{k-2} \gamma^{r-1} \sum_{q=0}^{r-2}  
 \sum_{I\in \mathcal{C}_{r,k}^{(q)}} 
\sum_{s=1}^{t^{\star}(\tI)} (\Gamma(1-\alpha/2))^{-r+q+1-s}\\
& \quad  \sum_{T\in \mathcal{C}_{s,|\tI|}(\tI)} \left(\prod_{i=1}^{r-q} \frac{\Gamma(d_i(\tI,T))}{\Gamma(N_i(\tI))} \right)
 \prod_{(i,t)\in \Delta(\tI,T)} \Gamma\Big(\frac{m_{it}(\tI,T)-\alpha}{2}\Big)\,.
\end{split}
\end{equation*}
For $r-q\ge 2$ let $\tI \in \mathcal{C}_{r-q,|\tI|}$ and $T\in \mathcal{C}_{s,|\tI|}(\tI)$. By definition of the set $\mathcal{C}_{s,|\tI|}(\tI)$, the number of distinct edges of $\Delta(\tI,T)$ is $N_e(\tI,T)= r-q+s-1$. Since $\tI$ is not totally reducible we have $m_{it}(\tI,T)\ge 4$ for at least one $(i,t)\in \Delta(\tI,T)$. Using these two facts we see for $\alpha<2$ sufficiently close to $2$ that 
\begin{equation}\label{eq:sdgfsds}
\prod_{(i,t)\in \Delta(\tI,T)} \Gamma\Big(\frac{m_{it}(\tI,T)-\alpha}{2}\Big)
\le c_{|\tI|} \,(\Gamma(1-\alpha/2))^{r-q+s-2}\,,
\end{equation}
where $c_{|\tI|}>0$ is a constant only depending on $|\tI|$. Therefore we have
\begin{equation*}
    \begin{split}
0\le \lim_{\alpha \to 2^-} d_k(\alpha,\gamma)&\le  \lim_{\alpha \to 2^-} \frac{c_k}{\Gamma(1-\alpha/2)} \sum_{r=2}^{k-2} \gamma^{r-1} \sum_{q=0}^{r-2} 
 \sum_{I\in \mathcal{C}_{r,k}^{(q)}} 
\sum_{s=1}^{t^{\star}(\tI)} \sum_{T\in \mathcal{C}_{s,|\tI|}(\tI)} \left(\prod_{i=1}^{r-q} \frac{\Gamma(d_i(\tI,T))}{\Gamma(N_i(\tI))} \right)\\
&=0\,.
\end{split}
\end{equation*}
\end{proof}

\begin{lemma}\label{lem:poisson}
Let $Z$ be a random variable  with moments
$$\E[Z^k]=\frac{1}{\gamma} \sum_{r=1}^{k} \gamma^{r} B(k,r)\,, \quad k\ge 1\,.$$
Then $Z$ follows a modified Poisson$(\gamma)$ distribution defined by
$$\P(Z=0)= 1-\frac1\gamma +\frac1\gamma \e^{-\gamma}   \qquad \text{ and } \qquad
\P(Z=k)=\e^{-\gamma}\frac{\gamma^{k-1}}{k!}\,,\qquad k\ge 1\,.$$
\end{lemma}
\begin{proof}
We compute the moment generating function of $Z$. Note that $B(k,0)=0$ for $k\ge 1$ and $B(0,0)=1$. Since $\E[Z^0]=1$ we have
\begin{equation*}
\begin{split}
\E[\e^{tZ}]&= \sum_{k=0}^{\infty} \frac{t^k}{k!}\E[Z^k]=1+ \sum_{k=1}^{\infty} \frac{t^k}{k!} \frac{1}{\gamma} \sum_{r=1}^{k} \gamma^{r} B(k,r)\\
&= 1-\frac{1}{\gamma} + \frac{1}{\gamma}\sum_{k=0}^{\infty} \frac{t^k}{k!}  \sum_{r=0}^{k} \gamma^{r} B(k,r)\\
&= 1-\frac{1}{\gamma} + \frac{1}{\gamma}\sum_{k=0}^{\infty} \frac{t^k}{k!}  T_k(\gamma)
\end{split}
\end{equation*}
where $T_k(\gamma):=\sum_{r=0}^{k} \gamma^{r} B(k,r)$ is the $k$-th Touchard polynomial which satisfy the identity 
$$\sum_{k=0}^{\infty} \frac{t^k}{k!}  T_k(\gamma) =
\e^{\gamma(\e^t-1)}\,.$$
This is the moment generating function
of a Poisson$(\gamma)$ distributed random variable $W$.
Therefore, 
\begin{equation*}
\E[\e^{tZ}]=   1-\frac{1}{\gamma} +\frac{1}{\gamma}\e^{\gamma(\e^t-1)}
          =1-\frac{1}{\gamma} +\frac{1}{\gamma}  \E[\e^{tW}].
\end{equation*}
In particular for $\gamma=1$,  $Z$ is Poisson$(\gamma)$ distributed.
In general, $Z$ follows the distribution,
$$\P(Z=0)= 1-\frac1\gamma +\frac1\gamma \e^{-\gamma}   \qquad \text{ and } \qquad
\P(Z=k)=\e^{-\gamma}\frac{\gamma^{k-1}}{k!}\,,\qquad k\ge 1\,.$$
Note that $1-\frac{1}{\gamma} +\frac{1}{\gamma}  e^{-\gamma}\ge0$ for
all $\gamma>0$. 
The proof is complete.
\end{proof}

%
%

\begin{funding}
J. Heiny was supported by the Deutsche Forschungsgemeinschaft (DFG) through RTG 2131 High-dimensional Phenomena in Probability – Fluctuations and Discontinuity. J. Yao's research was supported by  the HKSAR RGC grant GRF-17306918.
\end{funding}



\bibliographystyle{imsart-number} 
\bibliography{libraryFeb2020}       


\end{document}